\documentclass{article}
\usepackage{amsmath,amssymb,amsthm,fullpage,mathrsfs,pgf,tikz,caption,subcaption,mathtools}
\usepackage[colorinlistoftodos]{todonotes}
\usetikzlibrary{decorations.pathreplacing}
\usepackage[utf8]{inputenc}
\title{Doppelgangers: the Ur-Operation and Posets of Bounded Height}
\author{Thomas Browning, Max Hopkins, Zander Kelley}
\date{July 2018}
\newtheorem{theorem}{Theorem}[section]
\newtheorem{corollary}[theorem]{Corollary}
\newtheorem{proposition}[theorem]{Proposition}
\newtheorem{lemma}[theorem]{Lemma}
\newtheorem{definition}[theorem]{Definition}
\newtheorem{conjecture}[theorem]{Conjecture}

\newtheorem{example}[theorem]{Example}

\begin{document}
\maketitle
\begin{abstract}In the early 1970s, Richard Stanley and Kenneth Johnson introduced and laid the groundwork for studying the order polynomial of partially ordered sets (posets). Decades later, Hamaker, Patrias, Pechenik, and Williams introduced the term ``doppelgangers": equivalence classes of posets given by equality of the order polynomial. We provide necessary and sufficient conditions on doppelgangers through application of both old and novel tools, including new recurrences and the Ur-operation: a new generalized poset operation. In addition, we prove that the doppelgangers of posets P of bounded height $|P|-k$ may be classified up to systems of $k$ diophantine equations in $2^{O(k^2)}$ time, and similarly that the order polynomial of such posets may be computed in $O(|P|)$ time. An extended abstract of this paper appears in Issue 80B of Séminaire Lotharingien Combinatoire.
\end{abstract}
\section{Introduction}
\subsection{Background}
Richard Stanley introduced the order polynomial $F_P(m)$ of an unlabeled partially ordered set (poset) in 1970 as an analog to chromatic polynomials \cite{Order}. Soon after, Johnson introduced a recurrence relation on the order polynomial of unlabeled posets \cite{Johnson} which Stanley expanded upon through the introduction of induction on incomparable elements, a powerful tool for studying posets. Computing the order polynomial is difficult. For instance, Brightwell and Winkler proved that computing even the first coefficient of the order polynomial (counting linear extensions) is \#P-complete \cite{complexity}. Despite this, Faigle and Schrader proved that the order polynomial of special families, series-parallel posets and posets of bounded (constant) width, may be computed in polynomial time \cite{Faigle}. 

More recently, Boussicault, Feray, Lascoux, and Reiner examined posets from a geometric perspective by studying linear extension sums as valuations over polyhedral cones \cite{Feray2}. In their work, the authors re-introduce induction on incomparable elements, extending a simple recurrence on linear extensions to valuations. In 2014, McNamara and Ward \cite{McNamara} set out to classify the equivalence classes of the multivariate generating function $K_{(P,\omega)}$, a function introduced by Gessel in 1983 \cite{Ges}, and closely related to the labeled order polynomial $\Omega_{P,\omega}(m)$. In their work, McNamara and Ward prove a number of important poset invariants for $K_{P,\omega}(m)$, and offer several conjectures and unexplained equivalences--one of which we explain in section \ref{sp-results}. Later, Hamaker, Patrias, Pechenik and Williams coined the term doppelgangers for unlabeled posets with the same order polynomial, and demonstrated several examples related to the K-theory of miniscule varieties \cite{Hamaker}. Their paper focuses on infinite families of grid-like doppelgangers, raising the natural question of the existence and importance of similar families. We apply Johnson's initial recurrence to $F_P(m)$ as well as a new recurrence on both $\Omega_{P,\omega}(m)$ and $K_{(P,\omega)}$ similar to that used in \cite{Feray2} in order to further study doppelgangers. 
\subsection{Results}
\label{sec:results}
Our work begins with an exploration of the interaction between doppelgangers and the standard poset operations disjoint union and ordinal sum, the operations used to build series-parallel posets.
To this end, we introduce a number of recurrences that require the following definitions.
%To this end, we re-introduce Johnson's unlabeled recurrence and consider novel recurrences on $\Omega_{P,\omega}$ and $K_{P,\omega}$, which are closer in form to that used in \cite{Feray2}.
For incomparable elements $x,y$, let $P|x\leq y$ be the poset with added cover relation $x \lessdot y$ and all further relations required by transitivity, and $P|x=y$ be $P$ with $x$ and $y$ identified. In particular, if $v$ is the identification of $x$ and $y$ then $z\leq v$ in $P|x=y$ if and only if $z\leq x$ in $P$ or $z\leq y$ in $P$. Finally, given a labeled poset $(P,\omega)$, let $(P,\omega)|x<y$ be the poset $(P,\omega)$ with the added \textit{strict} relation $x<y$ and all other relations implied by transitivity. Note that $(P,\omega)|x<y$ may not correspond to a labeled poset. 

Recall that $F_P(m)$ counts the number of order-preserving maps $P\to[m]=\{1,\ldots,m\}$. $\Omega_{P,\omega}(m)$ counts the number of $(P,\omega)$-partitions into $[m]$--order preserving maps $P\to[m]$ that are consistent with the labeling $\omega$ (a bijection between $P$ and $[|P|]$). Finally, $K_{P,\omega}(x)$ is a sum over all $(P,\omega)$-partitions $f$ of the product of $x_j^{|f^{-1}(j)|}$ for each $j\geq1$.
More detail regarding these definitions can be found at the end of Section 2.
\begin{lemma}
\label{Irec}
The order polynomial and multivariate generating function admit the following recurrences:
\begin{align}
F_P &=F_{P|x\leq y}+F_{P|y\leq x}-F_{P|x=y}\\
\Omega_{P,\omega} &=\Omega_{P|x\leq y,\omega}+\Omega_{P|y\leq x,\omega}\\
K_{P,\omega} &=K_{P|x<y,\omega}+K_{P|y\leq x,\omega}
\end{align}
\end{lemma}
The objects $P|x<y,\omega$ and $P|y\leq x,\omega$ in $(1.3)$ might not be posets, but these objects remain valid for the purpose of calculating $K_{P,\omega}$. While we mostly focus on these recurrences to examine ordinal sum, they provide further results on doppelgangers as well. For instance, just a single step of recurrence $(1.1)$ provides new infinite families.
\begin{example}
\label{ex:fam_1}
For each $n\geq 2$, the posets $P_1$ and $P_2$ below are doppelgangers.
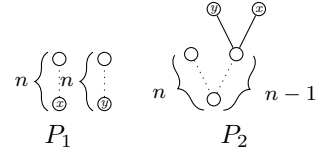
\begin{center}
\begin{tabular}{c c}
\begin{tikzpicture}[scale=0.3]
\node[draw, circle, inner sep=0pt, minimum size=5pt] at (0,0) (a) {\tiny x};
\node[draw, circle, inner sep=0pt, minimum size=5pt] at (0,2) (b) {};
\node[draw, circle, inner sep=0pt, minimum size=5pt] at (2,0) (c) {\tiny y};
\node[draw, circle, inner sep=0pt, minimum size=5pt] at (2,2) (d) {};
\draw[dotted] (a) -- (b);
\draw[dotted] (c) -- (d);
\draw [decorate,decoration={brace,amplitude=5pt}]
(-0.5,-0.3) -- (-0.5,2.3) node [black,midway,xshift=-0.35cm]
{\footnotesize $n$};
\draw [decorate,decoration={brace,amplitude=5pt}]
(1.5,-0.3) -- (1.5,2.3) node [black,midway,xshift=-0.35cm]
{\footnotesize $n$};
\end{tikzpicture}
&\begin{tikzpicture}[scale=0.3]
\node[draw, circle, inner sep=0pt, minimum size=5pt] at (0,0) (a) {};
\node[draw, circle, inner sep=0pt, minimum size=5pt] at (1,2) (b) {};
\node[draw, circle, inner sep=0pt, minimum size=5pt] at (-1,2) (c) {};
\node[draw, circle, inner sep=0pt, minimum size=5pt] at (2,4) (f) {\tiny x};
\node[draw, circle, inner sep=0pt, minimum size=5pt] at (0,4) (g) {\tiny y};
\draw[dotted] (c) -- (a) -- (b);
\draw (g) -- (b) -- (f);
\draw [decorate,decoration={brace,amplitude=5pt},rotate=9]
(-0.5,-0.4) -- (-1.4,2.2) node [black,midway,xshift=-0.4cm,yshift=-.15cm]
{\footnotesize $n$};
\draw [decorate,decoration={brace,mirror,amplitude=5pt},rotate=-10]
(0.5,-0.4) -- (1.4,2.2) node [black,midway,xshift=0.7cm,yshift=-.15cm]
{\footnotesize $n-1$};
\end{tikzpicture}
\\$P_1$
&$P_2$
\end{tabular}
\end{center}
We have the isomorphisms $(P_1|x\leq y)\cong(P_2|x\leq y)$, $(P_1|y\leq x)\cong(P_2|y\leq x)$, and $(P_1|x=y)\cong(P_2|x=y)$.
Since isomorphic posets have the same order polynomial, Equation 1.1 from Lemma 1.1 shows that $F_{P_1}=F_{P_2}$. Setting $n = 2$, we recover the Nicomachus formula
\[
\sum\limits_{k=1}^m k^3 = F_{P_2} = F_{P_1} = \left (\sum\limits_{k=1}^m k \right )^2
\]
\end{example}

In their work, McNamara and Ward offer four pairs of posets with equivalent $K_{P,\omega}$ which their methods do not explain as a springboard for further investigation \cite{McNamara}. Our improper recurrence, Equation (1.3), easily shows the first of these pairs, given in Figure \ref{fig:MW}, have equivalent $K_{P, \omega}$.
We expect Lemma \ref{Irec} has far reaching consequences for $K_{P,\omega}$.

In order to study the interaction of doppelgangers and the ordinal sum, we combine these recurrences with Stanley's method of induction on incomparable elements \cite{Stanley1}, re-introduced recently in \cite{Feray2}. This method provides elegant proofs of old results such as Stanley's poset reciprocity theorem \cite{Order}, and provides a basis for the order polynomial which interacts well with ordinal sum (see Proposition \ref{basisformula}), leading to the following results. Recall that the ordinal sum of $P$ and $Q$, $P \oplus Q$, follows from stacking the Hasse diagrams of $P$ and $Q$. We say $P \sim Q$ when $F_P(x) = F_Q(x)$.
Using induction on incomarable elements, we show in Lemma 3.2 that the order polynomial of an ordinal sum $P\oplus Q$ is given by the Cauchy product of the order polynomials of $P$ and $Q$ in a basis of binomial coefficients.
As corollaries of Lemma 3.2, we have the following results.
\begin{corollary}
\label{I2-3}
For labeled posets $(P,\omega),(P^\prime,\omega^\prime),(Q,\psi),(Q^\prime,\psi^\prime)$, any two conditions imply the third:

1) $(P,\omega)\sim(P^\prime,\omega^\prime)$

2) $(Q,\psi)\sim(Q^\prime,\psi^\prime)$

3) $(P\oplus Q,\omega\oplus\psi)\sim(P^\prime\oplus Q^\prime,\omega^\prime\oplus\psi^\prime)$
\end{corollary}
\begin{corollary}
\label{Isym}
For all labeled posets $(P,\omega),(Q,\psi)$,
\[(P\oplus Q,\omega\oplus\psi)\sim(Q\oplus P,\psi\oplus\omega).\]
\end{corollary}

While Lemma \ref{Irec} and Corollaries \ref{I2-3} and \ref{Isym} explain a large number of small and series-parallel doppelgangers, there are examples of size $\geq 6$ (see Example \ref{Ur-ex}) they cannot explain. To this end, we introduce a new poset operation to generalize Corollaries \ref{I2-3} and \ref{Isym}. 
\begin{definition}
\label{Ur-def}
For a poset $\mathscr{P}=\{x_1,\cdots,x_n\}$ and a sequence of posets $\{P_1,\cdots,P_n\}$, let $\mathscr{P}[x_k\to P_k]_{k=1}^n$ be the poset on $\bigcup_kP_k$ with the following operation:
\[\text{For }p\in P_j,q\in P_k,\ p\leq q\text{ when }\begin{cases}p\leq q&j=k\\x_j\leq x_k&j\neq k\end{cases}.\]
We call this the Ur-operation on $\mathscr{P}$ by $\{P_1,\cdots,P_n\}$. If any $P_k$ is not specified, then that $P_k$ is assumed to be the poset on one element.
\end{definition}
\begin{example}
The Ur-operation generalizes disjoint union, ordinal sum, and ordinal product.
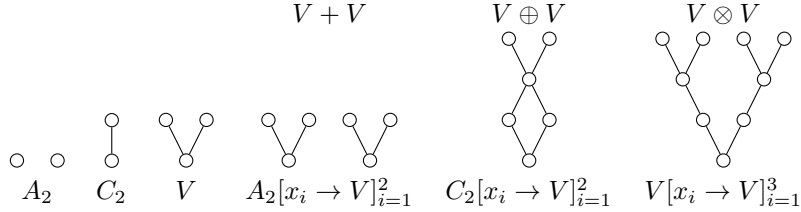
\begin{center}
\begin{tabular}{c c c c c c}
&
&
& $V + V$
& $V \oplus V$
& $V \otimes V$
\\
\begin{tikzpicture}[scale=0.27]
\node[draw, circle, inner sep=0pt, minimum size=5pt] at (-1,0) (a) {};
\node[draw, circle, inner sep=0pt, minimum size=5pt] at (1,0) (b) {};
\end{tikzpicture}
&\begin{tikzpicture}[scale=0.27]
\node[draw, circle, inner sep=0pt, minimum size=5pt] at (0,-1) (a) {};
\node[draw, circle, inner sep=0pt, minimum size=5pt] at (0,1) (b) {};
\draw (a) -- (b);
\end{tikzpicture}
&\begin{tikzpicture}[scale=0.27]
\node[draw, circle, inner sep=0pt, minimum size=5pt] at (1,0) (a) {};
\node[draw, circle, inner sep=0pt, minimum size=5pt] at (2,2) (b) {};
\node[draw, circle, inner sep=0pt, minimum size=5pt] at (0,2) (c) {};
\draw (b) -- (a) -- (c);
\end{tikzpicture}
&\begin{tikzpicture}[scale=0.27]
\node[draw, circle, inner sep=0pt, minimum size=5pt] at (1,0) (a) {};
\node[draw, circle, inner sep=0pt, minimum size=5pt] at (2,2) (b) {};
\node[draw, circle, inner sep=0pt, minimum size=5pt] at (0,2) (c) {};
\node[draw, circle, inner sep=0pt, minimum size=5pt] at (5,0) (d) {};
\node[draw, circle, inner sep=0pt, minimum size=5pt] at (6,2) (e) {};
\node[draw, circle, inner sep=0pt, minimum size=5pt] at (4,2) (f) {};
\draw (b) -- (a) -- (c);
\draw (e) -- (d) -- (f);
\end{tikzpicture}
&\begin{tikzpicture}[scale=0.27]
\node[draw, circle, inner sep=0pt, minimum size=5pt] at (1,0) (a) {};
\node[draw, circle, inner sep=0pt, minimum size=5pt] at (2,2) (b) {};
\node[draw, circle, inner sep=0pt, minimum size=5pt] at (0,2) (c) {};
\node[draw, circle, inner sep=0pt, minimum size=5pt] at (1,4) (d) {};
\node[draw, circle, inner sep=0pt, minimum size=5pt] at (2,6) (e) {};
\node[draw, circle, inner sep=0pt, minimum size=5pt] at (0,6) (f) {};
\draw (b) -- (a) -- (c);
\draw (e) -- (d) -- (f);
\draw (b) -- (d) -- (c);
\end{tikzpicture}
&\begin{tikzpicture}[scale=0.27]
\node[draw, circle, inner sep=0pt, minimum size=5pt] at (1,0) (a) {};
\node[draw, circle, inner sep=0pt, minimum size=5pt] at (2,2) (b) {};
\node[draw, circle, inner sep=0pt, minimum size=5pt] at (0,2) (c) {};
\node[draw, circle, inner sep=0pt, minimum size=5pt] at (3,4) (d) {};
\node[draw, circle, inner sep=0pt, minimum size=5pt] at (4,6) (e) {};
\node[draw, circle, inner sep=0pt, minimum size=5pt] at (2,6) (f) {};
\node[draw, circle, inner sep=0pt, minimum size=5pt] at (-1,4) (g) {};
\node[draw, circle, inner sep=0pt, minimum size=5pt] at (0,6) (h) {};
\node[draw, circle, inner sep=0pt, minimum size=5pt] at (-2,6) (i) {};
\draw (b) -- (a) -- (c);
\draw (e) -- (d) -- (f);
\draw (i) -- (g) -- (h);
\draw (b) -- (d);
\draw (c) -- (g);
\end{tikzpicture}
\\$A_2$
& $C_2$
&$V$
&$A_2[x _i\to V]_{i=1}^2$
&$C_2[x _i\to V]_{i=1}^2$
&$V[x_i \to V]_{i=1}^3$
\\
\end{tabular}
\end{center}
\end{example}
Further, using the operation we prove a generalization of Corollary \ref{I2-3}:
\begin{theorem}
For a poset $\mathscr P=\{x_1,\cdots,x_n\}$ and two sequences of posets $\{P_1,\ldots,P_n\}$ and $\{Q_1,\ldots,Q_n\}$ such that $P_i\sim Q_i$, we have that $\mathscr P[x_k\to P_k]_{k=1}^n\sim\mathscr P[x_k\to Q_k]_{k=1}^n$.
\label{Ur-thm}
\end{theorem}
Theorem \ref{Ur-thm} shows that elements of the same poset may be exchanged for doppelgangers while preserving equivalence. This raises the natural question of when distinct elements may be exchanged with the same result.
\begin{definition}
We say $x \in P,\ y \in Q$ are Ur-equivalent when $P[x \to R] \sim Q[y \to S]$ for all posets $R\sim S$.
\end{definition}
In Corollary \ref{urequivalenceresult} and Conjecture \ref{urequivalenceconjecture}, we offer a necessary and sufficient condition for Ur-equivalence, and conjecture a strengthening of the result.

Finally, we move to the classification of infinite families of doppelgangers. Faigle and Schrader proved that for posets with bounded width $k$, the order polynomial may be computed in $O(|P|^{2k+1})$ time. However, any algorithm to classify infinite families of doppelgangers must be constant with respect to $|P|$. We provide such an algorithm for posets of height $|P|-k$, a subfamily of Faigle and Schrader's posets of bounded width.
\begin{theorem}
\label{class-time}
For constant k, the doppelgangers among posets of height $|P|-k=n-k$ are completely determined by sets of $k$ diophantine equations computable in $2^{O(k^2)}$ time. In addition, $F_P(x)$ is computable in $O(n)$ time, and for $k=O(\frac{log(n)}{log(log(n))})$, the time is polynomial in $n$. 
\end{theorem}
Theorem \ref{class-time} takes advantage of several invariants on doppelgangers we will introduce in Section \ref{sp-results}, as well as the rigid structure of posets of bounded height. The improvement this structure brings from $O(n^{2k+1})$ to $O(n)$ allows us to extend our family of bounded height past the constant restriction imposed by Faigle and Schrader on posets of bounded width. As an example, we provide the diophantine equations for $k=1,2$ in Table \ref{tab:dio}, along with general solutions where possible.

\section{Doppelgangers and the Order Polynomial \label{notation}}
For a poset $P$, let $F_P(n)$ denote the number of order-preserving maps $f$ from $P$ to $\{1,2,\ldots,n\}$ -- that is, maps which satisfy $f(x) \leq f(y)$ whenever $x \leq y$ in $P$. 
Thus the numbers $F_P(n)$ provide a measure of how far the poset $P$ is from a total order. 
If two posets $P$ and $Q$ satisfy the equivalence $F_P(n) = F_Q(n)$ for all $n$, we will call them \textit{doppelgangers}, and we denote this fact by $P \sim Q$.
In this paper we establish certain structural properties of a pair of of posets $(P,Q)$ which are either necessary or sufficient conditions for $P \sim Q$. Stanley offered many seminal necessary conditions in his early work and later as exercises in \textit{Enumerative Combinatorics} \cite{Stanley}. We provide some simple but important examples from these to aid intuition.
\begin{proposition}
If $P$ and $Q$ are doppelgangers, then they have the same number of elements.
\end{proposition}
\begin{proof}
Let $a_k$ be the number of \textit{surjective} order-preserving maps $f$ from $P$ to $\{1,2,\ldots,k\}$.
Then we have
$$ F_P(n) = \sum_{k=1}^{|P|} a_k \binom{n}{k}. $$

In particular, $F_P$ is a polynomial of degree $|P|$ (indeed, $F_P$ is called the \textit{order polynomial} of $P$ \cite{Stanley}).
\end{proof}

We recall several operations on posets and show that they behave well in relation to order polynomials. 
Let $P$ and $Q$ be posets, and let $\bf{1}$ denote the poset with a single element.

\begin{itemize}
\item The \textit{dual} of a poset $P$, denoted $P^*$, is constructed by reversing the direction of all relations in $P$.
\item The \textit{disjoint union} of $P$ and $Q$, denoted $P + Q$, is constructed by taking the union of the elements of $P$ and $Q$ and inheriting the relations from $P$ and $Q$ 
(thus the elements from $P$ remain incomparable with the elements from $Q$). 
For example, $\bf{1} + \bf{1} + \bf{1}$ is the anti-chain of size $3$, in which no two distinct elements are comparable. 
\item The \textit{ordinal sum} of $P$ and $Q$, denoted $P \oplus Q$, is constructed by first taking $P + Q$,
and then imposing the relation $x \leq y$ for every $x \in P$ and $y \in Q$. 
For example, $\bf{1} \oplus \bf{1} \oplus \bf{1}$ is the chain of size $3$, a total order.
\item The \textit{ordinal product} of $P$ and $Q$, denoted $P \otimes Q$, is constructed by taking the Cartesian product $P \times Q$ and imposing relations $(r, s) \leq (r', s')$ if $r<r'$ in $P$ or $r=r'$ in $P$ and $s \leq s'$ in $Q$. For example, $(\bf{1} + \bf{1} + \bf{1}) \otimes (\bf{1} \oplus \bf{1} \oplus \bf{1})$ is $(\bf{1} \oplus \bf{1} \oplus \bf{1}) + (\bf{1} \oplus \bf{1} \oplus \bf{1}) + (\bf{1} \oplus \bf{1} \oplus \bf{1})$
\end{itemize}

\begin{proposition}
$P \sim P^*$.
\end{proposition}

\begin{proof}
Consider the bijection which sends an order-preserving mapping $f : P \rightarrow \{1,2,\ldots ,n\}$ to the mapping $g$, where $g(x) \vcentcolon =   n + 1 - f(x)$. 
The mapping $g$ is order-preserving on $P^*$. 
Thus $F_P(n) = F_{P^*}(n)$.
\end{proof}

\begin{proposition}
$F_{P+Q}(n) = F_{P}(n)F_Q(n).$
\end{proposition}

\begin{proof}
Since the elements from $P$ and the elements from $Q$ are incomparable in $P+Q$, every choice of order-preserving maps $f$ on $P$ and $g$ on $Q$ gives rise to an-order preserving map
$$ h(x) \vcentcolon = \left\{\begin{array}{lr}
        f(x) & \text{if } x \in P\\
        g(x) & \text{if } x \in Q
        \end{array}\right\} $$
 on $P+Q$, and it is not hard to see that every order-preserving map on $P+Q$ is of this form.
\end{proof}

These operations can be used to generate larger, more complicated pairs of doppelgangers out of smaller pairs.
For example, if $Q \sim R$, then we get that
$$ F_{P+Q} = F_P F_Q = F_P F_R = F_{P+R},$$
and so $P + Q$ and $P + R$ are doppelgangers for all posets $P$.
Analogously, in Corollary 3.6 we will also see that $P \oplus Q$ and $P \oplus R$ are doppelgangers whenever $Q \sim R$. In fact, the Ur-operation provides a direct generalization of this property, given in Theorem \ref{blowuptheorem}. While we do not provide results on the ordinal product, the Ur-operation generalizes the ordinal product along with direct and ordinal sum.

The term doppelganger originally referred to unlabeled posets, but extends easily to labeled posets $(P,\omega)$. A \textit{labeled poset} $(P,\omega)$ is a poset $P$ equipped with a bijective labeling $\omega\colon P\to[|P|]$.
In this case, a map $f\colon(P,\omega)\to[m]$ is order-preserving when $f(x)\leq f(y)$ whenever $x\leq y$, and $f(x)<f(y)$ whenever $x<y$ and $\omega(x)>\omega(y)$. The number of such maps is the order polynomial of $(P,\omega)$, denoted $\Omega_{P,\omega}(m)$. In fact, every unlabeled poset $P$ may be written as a labeled poset $(P,\omega)$ where $\omega$ is a \textit{natural labeling} or a linear extension of $P$, that is when $\omega(x)<\omega(y)$ whenever $x<y$. In this case $F_P=\Omega_{P,\omega}$. In fact, labeled posets admit an interesting generalization of the order polynomial studied recently by McNamara and Ward \cite{McNamara}. The \textit{multivariate generating function} of $(P,\omega)$ is
\[
K_{P,\omega}(x) = \sum\limits_{f \in (P,\omega)-\text{partitions}} x_1^{|f^{-1}(1)|}x_2^{|f^{-1}(2)|} \ldots
\]
and related to $\Omega_{P,\omega}$ by 
\[
\Omega_{P,\omega}(m) = K_{P,\omega}(\underbrace{1,\ldots,1}_{m},0,\ldots).
\]
Here, $(P,\omega)$-partitions differ from order preserving maps only in that they map to the positive integers rather than $[m]$.
\section{Order Polynomial Recurrence}
\subsection{The Recurrence Relations \label{recurrences}}
We now formalize the recurrences given in Lemma \ref{Irec}. Given a poset $P$ with incomparable elements $x$ and $y$, we can define the poset $P|x\leq y$ to be the result of adding the cover relation $x\leq y$ and all other relations implied by transitivity.
We can define the poset $P|x=y$ to be the result of identifying $x$ and $y$. Further, note that labeled posets can be viewed as an assignment of strict and weak edges. This allows us to define $(P,\omega)|x<y$, $(P,\omega)$ with the added relation $x<y$ and all other relations implied by transitivity. This last restriction might not result in a valid labeled poset, but order preserving functions, and thus the order polynomial and multivariate generating functions are still well-defined on these improper posets. As an example, we offer the proof of recurrence (1.1) originally proposed by Johnson, the rest follow similarly.
\begin{example}
\[F_P=F_{P|x\leq y}+F_{P|y\leq x}-F_{P|x=y}.\]
\end{example}
\begin{proof}
In this relation, an order-preserving map $f\colon P\to[n]$ either has $f(x)<f(y)$ in which case it is counted by the first term, $f(x)>f(y)$ in which case it is counted by the second term, or $f(x)=f(y)$ in which case it is counted by all three terms.
\end{proof}
Equation (1.3) illuminates McNamara and Ward's first unexplained example as demonstrated in Figure \ref{fig:MW}. This ends our discussion of $K_{P,\omega}$, but application of our methodology to the function is a possible direction of further research.
\begin{figure}[htb]
\begin{center}
\begin{tabular}{c c c c c c}
\begin{tikzpicture}[scale=0.3]
\node[draw, circle, inner sep=0pt, minimum size=5pt] at (-1,0) (a) {\tiny x};
\node[draw, circle, inner sep=0pt, minimum size=5pt] at (1,0) (b) {\tiny y};
\node[draw, circle, inner sep=0pt, minimum size=5pt] at (-1,2) (c) {};
\node[draw, circle, inner sep=0pt, minimum size=5pt] at (1,2) (d) {};
\node[draw, circle, inner sep=0pt, minimum size=5pt] at (1,4) (e) {};
\draw (b) -- (c) -- (a) -- (d);
\draw (d) -- (e);
\draw[double] (b) -- (d);
\end{tikzpicture}
&\begin{tikzpicture}[scale=0.3]
\node[draw, circle, inner sep=0pt, minimum size=5pt] at (0,0) (a) {\tiny x};
\node[draw, circle, inner sep=0pt, minimum size=5pt] at (0,2) (b) {\tiny y};
\node[draw, circle, inner sep=0pt, minimum size=5pt] at (-1,4) (c) {};
\node[draw, circle, inner sep=0pt, minimum size=5pt] at (1,4) (d) {};
\node[draw, circle, inner sep=0pt, minimum size=5pt] at (1,6) (e) {};
\draw (b) -- (c);
\draw (d) -- (e);
\draw[double] (b) -- (d);
\draw[double] (a) -- (b);
\end{tikzpicture}
&\begin{tikzpicture}[scale=0.3]
\node[draw, circle, inner sep=0pt, minimum size=5pt] at (-1,2) (a) {\tiny x};
\node[draw, circle, inner sep=0pt, minimum size=5pt] at (0,0) (b) {\tiny y};
\node[draw, circle, inner sep=0pt, minimum size=5pt] at (-1,4) (c) {};
\node[draw, circle, inner sep=0pt, minimum size=5pt] at (1,4) (d) {};
\node[draw, circle, inner sep=0pt, minimum size=5pt] at (1,6) (e) {};
\draw (c) -- (a) -- (d);
\draw (d) -- (e);
\draw[double] (b) -- (d);
\draw (a) -- (b);
\end{tikzpicture}
&\begin{tikzpicture}[scale=0.3]
\node[draw, circle, inner sep=0pt, minimum size=5pt] at (0,0) (a) {};
\node[draw, circle, inner sep=0pt, minimum size=5pt] at (1,2) (b) {\tiny y};
\node[draw, circle, inner sep=0pt, minimum size=5pt] at (-1,2) (c) {\tiny x};
\node[draw, circle, inner sep=0pt, minimum size=5pt] at (1,4) (d) {};
\node[draw, circle, inner sep=0pt, minimum size=5pt] at (-1,4) (e) {};
\draw (e) -- (c);
\draw (a) -- (b) -- (d);
\draw[double] (a) -- (c);
\end{tikzpicture}
&\begin{tikzpicture}[scale=0.3]
\node[draw, circle, inner sep=0pt, minimum size=5pt] at (0,0) (a) {};
\node[draw, circle, inner sep=0pt, minimum size=5pt] at (0,2) (b) {\tiny x};
\node[draw, circle, inner sep=0pt, minimum size=5pt] at (-1,4) (c) {};
\node[draw, circle, inner sep=0pt, minimum size=5pt] at (1,4) (d) {\tiny y};
\node[draw, circle, inner sep=0pt, minimum size=5pt] at (1,6) (e) {};
\draw (b) -- (c);
\draw (d) -- (e);
\draw[double] (b) -- (d);
\draw[double] (a) -- (b);
\end{tikzpicture}
&\begin{tikzpicture}[scale=0.3]
\node[draw, circle, inner sep=0pt, minimum size=5pt] at (1,2) (a) {\tiny y};
\node[draw, circle, inner sep=0pt, minimum size=5pt] at (0,0) (b) {};
\node[draw, circle, inner sep=0pt, minimum size=5pt] at (1,4) (c) {};
\node[draw, circle, inner sep=0pt, minimum size=5pt] at (-1,4) (d) {\tiny x};
\node[draw, circle, inner sep=0pt, minimum size=5pt] at (-1,6) (e) {};
\draw (c) -- (a) -- (d);
\draw (d) -- (e);
\draw[double] (b) -- (d);
\draw (a) -- (b);
\end{tikzpicture}
\\$P$
&$P|x < y$
&$P|y \leq x$
&$Q$
&$Q|x < y$
&$Q|y \leq x$\\
\end{tabular}
\end{center}
\caption{Equivalence of $K_{P,\omega}$ and $K_{Q,\omega}$. Double edges denoted strict order relations.}
\label{fig:MW}
\end{figure}
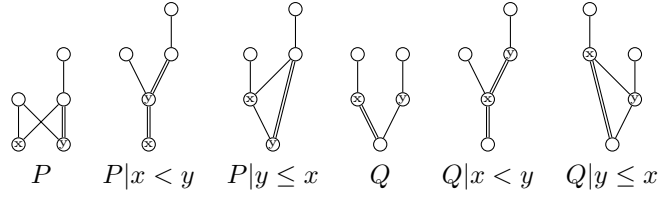

\subsection{Induction on Incomparable Elements \label{sp-results}}
We use the recurrences from Section \ref{recurrences} to prove results involving order polynomials by strong induction on the number of incomparable pairs of elements.
In particular, each of the terms of the recurrences have fewer pairs of incomparable elements than the original poset.
In his early work \cite{Stanley1}, Stanley uses this technique to prove the well-known expression for the strict order polynomial in the $\binom{m}{k}$ basis.
We will first introduce the power of this technique by providing a novel and short proof of Stanley's poset reciprocity theorem, and further offer an expression for the order polynomial of an ordinal sum of posets.

For a labeling $\omega$ of a poset $P$, let $\overline{\omega}$ be the dual labeling to $\omega$ given by $\overline{\omega}(x)=|P|+1-\omega(x)$.
In what follows, we will use the binomial reciprocity theorem, which states that
\[\binom{-n}{p}=\frac{(-n)(-n-1)\ldots(-n-(p-1))}{p!}=(-1)^p\frac{n(n+1)\ldots(n+p-1)}{p!}=(-1)^p\binom{n+p-1}{p}.\]
We provide this proof of Theorem \ref{reciprocity} as a simple introduction to how the recurrence relations are used in practice. In particular, a result is first proved for posets of the form $(C_k,\omega)$ and is then extended to all posets by strong induction on the number of pairs of incomparable elements and the recurrence relation.
\begin{theorem}[Poset Reciprocity]
\label{reciprocity}
For all labeled posets, $(P,\omega)$,
\[\Omega_{P,\overline{\omega}}(m)=(-1)^{|P|}\Omega_{P,\omega}(-m).\]
\end{theorem}
\begin{proof}
We shall proceed by strong induction on the number of pairs of incomparable elements in $P$.
For the base case where $P$ has no pairs of incomparable elements, $P$ is a chain.
Then $(P,\omega)$ can be thought to be a chain with $i$ strict edges and $j$ non-strict edges where $i+j=|P|-1$. 
Using a modified stars and bars technique, we get that $\Omega_{P,\omega}(m)=\binom{m+j}{|P|}$.
Since $(P,\overline{\omega})$ is a chain with $j$ strict edges and $i$ non-strict edges, $\Omega_{P,\overline{\omega}}(m)=\binom{m+i}{|P|}$.
Then by the binomial reciprocity theorem,
\[\Omega_{P,\overline{\omega}}(m)=\binom{m+i}{|P|}=(-1)^{|P|}\binom{-(m+i)+|P|-1}{|P|}=(-1)^{|P|}\binom{-m+j}{|P|}=(-1)^{|P|}\Omega_{P,\omega}(-m)\]
which shows the base case.
Now suppose that the result holds for all posets with fewer than $n$ pairs of incomparable elements and suppose that $P$ has $n$ pairs of incomparable elements.
Then let $x,y\in P$ be incomparable.
By our inductive assumption,
\begin{align*}
\Omega_{P,\overline{\omega}}(m)&=\Omega_{P|x\leq y,\overline{\omega}}(m)+\Omega_{P|y\leq x,\overline{\omega}}(m)\\
&=(-1)^{|P|}\Omega_{P|x\leq y,\omega}(-m)+(-1)^{|P|}\Omega_{P|y\leq x,\omega}(-m)\\
&=(-1)^{|P|}\Omega_{P,\omega}(-m)
\end{align*}
which shows the inductive step and completes the proof.
\end{proof}
It is clear from repeated applications Johnson's recurrence that the order polynomial of any poset should have an expression as the sum of the order polynomial of total orders, or \textit{chains}, with $F_{C_k} = \binom{m+k-1}{k}$ where $C_k$ is a chain of cardinality $k$. Indeed, as a consequence of poset reciprocity we can easily derive the expression for the order polynomial in the $\binom{m+k-1}{k}$ or \textit{chain basis}. 

\begin{proposition}
\label{basisformula}
For all posets $P$, there exist $c_k\in\mathbb{N}$ such that
\[F_P(m)=(-1)^{|P|}\sum_{k=h(P)}^{|P|}(-1)^kc_k\binom{m+k-1}{k}\]
where $h(P)$ is the height of $P$ and denotes the number of elements in the largest total order in $P$.
\end{proposition}
\begin{proof}
Let $\omega$ be a natural labeling for $P$.
By the poset and binomial reciprocity theorems, it suffices to show that there exist $c_k\in\mathbb{N}$ such that $\Omega_{P,\overline{\omega}}(m)=\sum_{k=h(P)}^{|P|}c_k\binom{m}{k}$.
It is straightforward to verify that we can let $c_k$ be the number of surjective strict order-preserving maps $f\colon P\to[k]$.
\end{proof}

In fact, in the chain basis, ordinal sum interacts with the order polynomial just as disjoint union interacts with the order polynomial in the standard basis. That is, the coefficients of $F_{P \oplus Q}$ in the chain basis are given by the convolution of the coefficients of $P$ with those of $Q$. Further, this extends to labeled posets and beyond the chain basis.
In particular, we generalize $\oplus$ to labeled posets in the following way:
given labeled posets $(P,\omega)$ and $(Q,\psi)$, let $\omega\oplus\psi$ be a labeling on $P\oplus Q$ given by
\[(\omega\oplus\psi)(x)=\begin{cases}\omega(x)&x\in P\\|P|+\psi(x)&x\in Q\end{cases}.\]
Then $(P\oplus Q,\omega\oplus\psi)$ is the labeled poset where every element of $P$ is weakly less than every element of $Q$.
The following result gives a formula for the order polynomial of an ordinal sum.
\begin{lemma}
\label{Llemma}
For all labeled posets $(P,\omega),(Q,\psi)$,
\[L(\Omega_{P\oplus Q,\omega\oplus\psi})=L(\Omega_{P,\omega})L(\Omega_{Q,\psi})\]
for any linear transformation $L$ on the polynomials in $m$ such that
\[L\left(\binom{m+c+d-1}{c+d}\right)=L\left(\binom{m+c-1}{c}\right)L\left(\binom{m+d-1}{d}\right)\]
for all integer $c,d\geq0$.
\end{lemma}
\begin{proof}
We first show the result in the case where $P$ and $Q$ are chains.
Suppose that $(P,\omega)$ has $i$ strict edges and $j$ non-strict edges and suppose that $(Q,\psi)$ has $k$ strict edges and $l$ non-strict edges.
Then $(P\oplus Q,\omega\oplus\psi)$ has $i+k+1$ strict edges and $j+l$ non-strict edges.
Then it suffices to show that
\[L\left(\binom{m+j}{|P|}\right)L\left(\binom{m+l}{|Q|}\right)=L\left(\binom{m+j+l+1}{|P|+|Q|}\right)\text{ where }j\leq|P|-1,\ l\leq|Q|-1.\]
We shall proceed by induction on $|P|-1-j+|Q|-1-l$.
For the base case of $|P|-1-j+|Q|-1-l=0$, $j=|P|-1$ and $l=|Q|-1$ in which case the result reduces to the hypothesis on $L$.
Now suppose that $|P|-1-j+|Q|-1-l>0$ and that the result holds for smaller values of $|P|-1-j+|Q|-1-l$.
Then without loss of generality, $j<|P|-1$ and
\begin{align*}
L\left(\binom{m+j}{|P|}\right)L\left(\binom{m+l}{|Q|}\right)&=L\left(\binom{m+j+1}{|P|}-\binom{m+j}{|P|-1}\right)L\left(\binom{m+l}{|Q|}\right)\\
&=L\left(\binom{m+j+1}{|P|}\right)L\left(\binom{m+l}{|Q|}\right)-L\left(\binom{m+j}{|P|-1}\right)L\left(\binom{m+l}{|Q|}\right)\\
&=L\left(\binom{m+j+l+2}{|P|+|Q|}\right)-L\left(\binom{m+j+l+1}{|P|+|Q|-1}\right)\\
&=L\left(\binom{m+j+l+1}{|P|+|Q|}\right)
\end{align*}
which shows the inductive step and completes the proof of the case where $P$ and $Q$ are chains.
For the general result, we shall proceed by strong induction on the number of pairs of incomparable elements in $P\oplus Q$.
For the base case where $P\oplus Q$ has no pairs of incomparable elements, $P$ and $Q$ are chains which was dealt with above.
Now suppose that the result holds for all posets where $P\oplus Q$ has fewer than $n$ pairs of incomparable elements and suppose that $P\oplus Q$ has $n$ pairs of incomparable elements.
Then without loss of generality, $P$ has an incomparable pair of elements, $x$ and $y$.
Then by the linearity of $L$ and our inductive assumption,
\begin{align*}
L(\Omega_{P\oplus Q,\omega\oplus\psi})&=L(\Omega_{(P\oplus Q)|x\leq y,\omega\oplus\psi}+\Omega_{(P\oplus Q)|y\leq x,\omega\oplus\psi})\\
&=L(\Omega_{(P|x\leq y)\oplus Q,\omega\oplus\psi})+L(\Omega_{(P|y\leq x)\oplus Q,\omega\oplus\psi})\\
&=L(\Omega_{P|x\leq y,\omega})L(\Omega_{Q,\psi})+L(\Omega_{P|y\leq x,\omega})L(\Omega_{Q,\psi})\\
&=L(\Omega_{P|x\leq y,\omega}+\Omega_{P|y\leq x,\omega})L(\Omega_{Q,\psi})\\
&=L(\Omega_{P,\omega})L(\Omega_{Q,\psi})
\end{align*}
which shows the inductive step and completes the proof.
\end{proof}
This implies our desired result in the chain basis.
\begin{corollary}
If $F_P(m) = \sum_{i=1}^{|P|} a_i {\binom{m + k - 1}{k}}$ and $F_Q(n) = \sum_{j=1}^{|Q|} b_j {\binom{m + k - 1}{k}}$, then
$$ F_{P\oplus Q}(n) = \sum_{k=1}^{|P|+|Q|} \left( \sum_{i=1}^k a_i b_{k-i} \right) {\binom{m + k - 1}{k}} .$$
Thus the coefficients of $F_{P\oplus Q}$ in the chain basis are given by the Cauchy product of the coefficients of $F_P$ and the coefficients of $F_Q$ in the chain basis.
\end{corollary}
\begin{proof}
Since the polynomials $\binom{m+k-1}{k}$ form a basis for the polynomials in $m$, the linear transformation sending $\binom{m+k-1}{k}\mapsto x^k$ satisfies the requirements of Lemma~\ref{Llemma}.
\end{proof}
Further, Lemma \ref{Llemma} provides immediate results on doppelgangers with respect to ordinal sum.
\begin{corollary}
For labeled posets $(P,\omega),(P^\prime,\omega^\prime),(Q,\psi),(Q^\prime,\psi^\prime)$, any two conditions imply the third:

1) $(P,\omega)\sim(P^\prime,\omega^\prime)$

2) $(Q,\psi)\sim(Q^\prime,\psi^\prime)$

3) $(P\oplus Q,\omega\oplus\psi)\sim(P^\prime\oplus Q^\prime,\omega^\prime\oplus\psi^\prime)$
\label{twoimplythird}
\end{corollary}
\begin{corollary}
\label{opluscommute}
For all labeled posets $(P,\omega),(Q,\psi)$,
\[(P\oplus Q,\omega\oplus\psi)\sim(Q\oplus P,\psi\oplus\omega).\]
\end{corollary}
Note that if $\omega$ is a natural labeling on $P$ and if $\psi$ is a natural labeling on $Q$ then $\omega\oplus\psi$ is a natural labeling on $P\oplus Q$.
As a consequence, Corollaries 2.8 and 2.9 also hold for unlabeled posets.
In Section 4, we will extend the unlabeled (naturally labeled) version of Corollary \ref{I2-3} to the Ur-operation. Despite its simplicity, Corollary \ref{Isym} has merit on its own, and easily recovers one of the doppelganger pairs discussed in \cite{Hamaker}.
If we let $C_n$ denote the total order on $n$ elements and if we let $A_n$ denote a collection of $n$ elements with no relations between them, then the posets in this example can be expressed by the ordinal sum as follows.
\begin{example}
Hamaker et al. show $C_{n-1} \oplus A_2 \oplus C_{n-1} \sim A_2 \oplus C_{n-1} \oplus C_{n-1}$ (see Figure 1 in \cite{Hamaker}). This pair of doppelgangers is an immediate consequence of Corollary \ref{opluscommute}.
\end{example}

\section{The Ur-Operation}
Section \ref{sp-results} details the interactions of the order polynomial and standard poset operations. By considering a generalization of these operations, it is possible in turn to extend our results. The operation itself is simple: consider replacing some subset of points in a poset $\mathscr P$ by a corresponding set of posets $\{P_1,\cdots,P_k\}$.
\begin{definition}
For a poset $\mathscr{P}=\{x_1,\cdots,x_n\}$ and a sequence of posets $\{P_1,\cdots,P_n\}$, let $\mathscr{P}[x_k\to P_k]_{k=1}^n$ be the poset on $\bigcup_kP_k$ with the following operation:
\[\text{For }p\in P_j,q\in P_k,\ p\leq q\text{ when }\begin{cases}p\leq q&j=k\\x_j\leq x_k&j\neq k\end{cases}.\]
We denote this as the Ur-operation on $\mathscr{P}$ by $\{P_1,\cdots,P_n\}$. All $P_k$ are assumed to be $C_1$ if not specified.
\end{definition}
Note that the disjoint sum operation denoted by $P_1+P_2$ can be expressed as $A_2[x_k\to P_k]_{k=1}^2$, the ordinal sum operation denoted by $P_1\oplus P_2$ can be expressed as $C_2[x_k\to P_k]_{k=1}^2$, and the ordinal product can be expressed as $P[x_k\to Q]_{k=1}^n$. 

The order polynomial of the Ur-operation relies heavily on the structure $\mathscr P$. Therefore it is convenient throughout the rest of this section to have the following definition
\begin{definition}
For a poset $P$ and $x\in P$, define $g_x^P(n,m)$ to be the number of order-preserving maps $f: P[x \to \varnothing] \to [m]$ such that there are exactly $n$ ways to extend $f$ to an order preserving map $\widetilde{f}: P \to [m].$
\end{definition}
\begin{example}
For a chain $C_7$ and its $4$th smallest (middle) element $e_4$:
\[
g^{C_7}_{e_4}(n,m) = \sum\limits_{i=1}^{m-n} F_{C_{2}}(i)F_{C_{2}}(m-i-n+2)
\]
\end{example}
\subsection{The Order Polynomial}
With this in hand, we offer a simple formula for the order polynomial of a single substitution. The polynomial for the general operation may be given by repeated application
\begin{proposition}
\label{singleblowupformula}
For a poset $\mathscr{P}$ with $x\in\mathscr{P}$, a poset $Q$, and $m\geq1$,
\[F_{\mathscr{P}[x\to Q]}(m)=\sum_{n=1}^m g^{\mathscr P}_x(n,m)F_Q(n).\]
\end{proposition}
\begin{proof}
The result follows from summing over all order preserving functions $f: \mathscr P[x \to \varnothing] \to [m]$ and counting at each step the possible order preserving functions on Q that satisfy the arrangement.
\end{proof}
As expected, the formulae for direct and ordinal sum follow immediately from our generalization:
\begin{corollary}
$F_{P + Q}(m) = F_P(m)F_Q(m)$, and $F_{P \oplus Q} = \sum\limits_{i=1}^{m}F_Q(m+1-i)(F_P(i) - F_P(i-1))$ where $F_P(0)$ is defined to be 0.
\end{corollary}
\begin{proof}
By Proposition~\ref{singleblowupformula},
\begin{align*}
F_{P+Q}(m)&=F_{(P+e)[e\to Q]}(m)\\
&=\sum_{n=1}^mg_e^{P+e}(n,m)F_Q(n)\\
&=F_P(m)F_Q(m).
\end{align*}
where we leverage $g_e^{P+e}(n,m)=\begin{cases}F_P(m)&n=m\\0&n\neq m\end{cases}$

\noindent Similarily,
\begin{align*}
F_{P\oplus Q}(m)&=F_{(P\oplus e)[e\to Q]}(m)\\
&=\sum_{n=1}^mg_e^{P\oplus e}(n,m)F_Q(n)\\
&=\sum_{n=1}^m(F_P(m+1-n)-F_P(m-n))F_Q(n)
\end{align*}
where the result follows by replacing $n$ by $i=1+m+1-n$.
\end{proof}
Moreover, the following result shows that the Ur-operation generalizes the nice relation between ordinal sum and doppelgangers given by Corollaries \ref{twoimplythird} and \ref{opluscommute}. In particular, we have Theorem 1.6 which we restate below.
\begin{theorem}
\label{blowuptheorem}
For a poset $\mathscr P=\{x_1,\cdots,x_n\}$ and two sequences of posets $\{P_1,\ldots,P_n\}$ and $\{Q_1,\ldots,Q_n\}$ such that $P_i\sim Q_i$, we have that $\mathscr P[x_k\to P_k]_{k=1}^n\sim\mathscr P[x_k\to Q_k]_{k=1}^n$.
\end{theorem}
\begin{proof}
For a poset $P$, let $S_P(n)$ denote the number of strict surjective order preserving maps $f\colon P\to[n]$.
By Proposition~\ref{basisformula} it suffices to show that $S_{\mathscr P[x_k\to P_k]_{k=1}^n}=S_{\mathscr P[x_k\to Q_k]_{k=1}^n}$. We call an collection of intervals $\{[a_k,b_k]\}_{k=1}^n$ nice if they cover $[n]$ and if $b_j<a_k$ whenever $x_j<x_k$. Let $\mathscr A$ denote the set of nice collections of intervals.
Then
\[S_{\mathscr P[x_k\to P_k]_{k=1}^n}=\sum_{\mathscr A}\prod_{k=1}^nS_{P_k}(b_k-a_k+1)=\sum_{\mathscr A}\prod_{k=1}^nS_{Q_k}(b_k-a_k+1)=S_{\mathscr P[x_k\to Q_k]_{k=1}^n}.\]
where the middle equality uses the fact that $S_{P_k}=S_{Q_k}$ which follows from Proposition~\ref{basisformula} and the fact that a representation in the $\binom{x+m-1}{m}$ basis is unique.
\end{proof}
For ease of computation, note that we need only compute $S_{P_k}$ for nice intervals such that $h(P_k)=h(Q_k)\leq b_k-a_k+1\leq|P_k|=|Q_k|$. If any $b_k-a_k+1<h(P_k)=h(Q_k)$ or $b_k-a_k+1>|P_k|=|Q_k|$ then the corresponding term in the product will be 0.
\begin{example}
The posets in Figures \ref{Ur-ex}(c) and \ref{Ur-ex}(f) are doppelgangers by Theorem~\ref{blowuptheorem}. Due to the underlying non-series-parallel structure of $P$, this does not follow from Corollaries \ref{twoimplythird} or \ref{opluscommute}, nor does it follow from a single application of Johnson's recurrence.
\end{example}
We should mention that Theorem~\ref{blowuptheorem} can also be proved with the recurrence.
In particular, since the recurrence commutes with the Ur-operation we can perform a full chain decomposition on each $P_i$ and each $Q_i$ independently.
This will reduce the order polynomial of the $\mathscr P[x_k\to P_k]_{k=1}^n$ and $\mathscr P[x_k\to Q_k]_{k=1}^n$ to a sum of order polynomials of posets where each point of $\mathscr P$ is replaced by a chain.
Since each $P_i\sim Q_i$, these resulting posets will be isomorphic. This proof generalizes to labeled posets. More precisely, the Ur-operation generalizes to labeled posets and the recurrences for labeled posets allow this proof technique to generalize to labeled posets and the associated multivariate generating function.
\begin{figure}[htb]
\begin{center}
\begin{tabular}{c c c}
\begin{tikzpicture}[scale=0.5]
\node[draw, circle, inner sep=0pt, minimum size=5pt] at (0,0) (a) {y};
\node[draw, circle, inner sep=0pt, minimum size=5pt] at (0,2) (b) {};
\node[draw, circle, inner sep=0pt, minimum size=5pt] at (2,2) (c) {x};
\node[draw, circle, inner sep=0pt, minimum size=5pt] at (2,0) (d) {};
\draw (a) -- (b) -- (d) -- (c);
\end{tikzpicture}
&\begin{tikzpicture}[scale=0.5]
\node[draw, circle, inner sep=0pt, minimum size=5pt] at (1,0) (a) {};
\node[draw, circle, inner sep=0pt, minimum size=5pt] at (2,2) (b) {};
\node[draw, circle, inner sep=0pt, minimum size=5pt] at (0,2) (c) {};
\draw (b) -- (a) -- (c);
\end{tikzpicture}
&\begin{tikzpicture}[scale=0.5]
\node[draw, circle, inner sep=0pt, minimum size=5pt] at (0,0) (a) {};
\node[draw, circle, inner sep=0pt, minimum size=5pt] at (0,2) (b) {};
\node[draw, circle, inner sep=0pt, minimum size=5pt] at (2,2) (c) {};
\node[draw, circle, inner sep=0pt, minimum size=5pt] at (2,0) (d) {};
\node[draw, circle, inner sep=0pt, minimum size=5pt] at (3,4) (e) {};
\node[draw, circle, inner sep=0pt, minimum size=5pt] at (1,4) (f) {};
\node[draw, circle, inner sep=0pt, minimum size=5pt] at (-1,-2) (g) {};
\node[draw, circle, inner sep=0pt, minimum size=5pt] at (1,-2) (h) {};
\draw (h) -- (a) -- (g);
\draw (a) -- (b) -- (d) -- (c);
\draw (c) -- (e);
\draw (c) -- (f);
\end{tikzpicture}
\\(a) $P$
&(b) $Q$
&(c) $P[x \to Q, y \to Q^*]$\\
\begin{tikzpicture}[scale=0.5]
\node[draw, circle, inner sep=0pt, minimum size=5pt] at (0,0) (a) {y};
\node[draw, circle, inner sep=0pt, minimum size=5pt] at (0,2) (b) {};
\node[draw, circle, inner sep=0pt, minimum size=5pt] at (2,2) (c) {x};
\node[draw, circle, inner sep=0pt, minimum size=5pt] at (2,0) (d) {};
\draw (a) -- (b) -- (d) -- (c);
\end{tikzpicture}
&\begin{tikzpicture}[scale=0.5]
\node[draw, circle, inner sep=0pt, minimum size=5pt] at (0,0) (a) {};
\node[draw, circle, inner sep=0pt, minimum size=5pt] at (2,0) (b) {};
\node[draw, circle, inner sep=0pt, minimum size=5pt] at (1,2) (c) {};
\draw (a) -- (c) -- (b);
\end{tikzpicture}
&\begin{tikzpicture}[scale=0.5]
\node[draw, circle, inner sep=0pt, minimum size=5pt] at (1,0) (a) {};
\node[draw, circle, inner sep=0pt, minimum size=5pt] at (-1,0) (g) {};
\node[draw, circle, inner sep=0pt, minimum size=5pt] at (0,-2) (h) {};
\node[draw, circle, inner sep=0pt, minimum size=5pt] at (0,2) (b) {};
\node[draw, circle, inner sep=0pt, minimum size=5pt] at (2,4) (c) {};
\node[draw, circle, inner sep=0pt, minimum size=5pt] at (2,0) (d) {};
\node[draw, circle, inner sep=0pt, minimum size=5pt] at (3,2) (e) {};
\node[draw, circle, inner sep=0pt, minimum size=5pt] at (1,2) (f) {};
\draw (a) -- (h) -- (g) -- (b);
\draw (a) -- (b) -- (d);
\draw (d) -- (e) -- (c) -- (f) -- (d);
\end{tikzpicture}
\\(d) $P$
&(e) $Q^*$
&(f) $P[x \to Q^*,y \to Q]$
\end{tabular}
\end{center}
\caption{An example of dopplegangers that are indecomposable under $+$ and $\oplus$ due to the Ur-operation.}
\label{Ur-ex}
\end{figure}
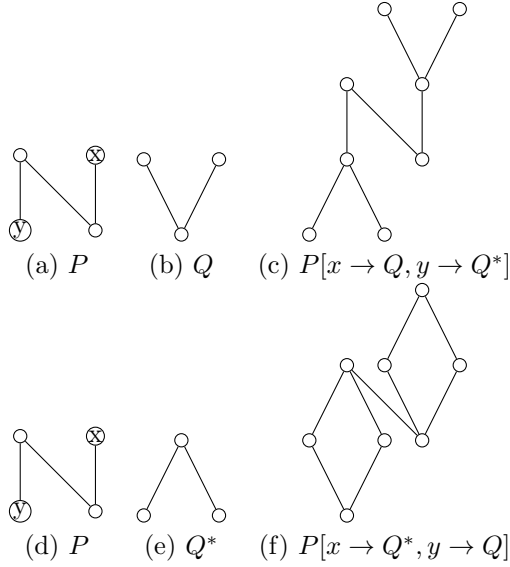
Theorem~\ref{blowuptheorem} allows us to build new doppelgangers out of an arbitrary poset by iteratively replacing points with corresponding doppelgangers. We know as well, however, that one can construct doppelgangers by replacing different points of some poset P with corresponding doppelgangers, such as in $C_k$ or $A_k$. It is natural then to ask about a generalization of this occurrence. For posets $P$ and $Q$ with $x\in P$ and $y\in Q$, when do we have $P[x \to R] \sim Q[y \to S]$ for all doppelgangers $R\sim S$?
\subsection{Ur-Equivalence}
\begin{definition}
We say $x \in P,\ y \in Q$ are Ur-equivalent when $P[x \to R] \sim Q[y \to S]$ for all posets $R\sim S$.
\end{definition}
In fact, Ur-equivalence relies on exactly the same structure the order polynomial does: on the values $g^P_x(n,m)$ and $g^Q_y(n,m)$.

\begin{proposition}
For $x \in P$ and $y \in Q$, $x$ and $y$ are Ur-equivalent if and only if $g^P_x = g^Q_y$
\end{proposition}
\begin{proof}
The backward direction is immediate from Proposition~\ref{singleblowupformula}. If $x \in P, y \in Q$ are Ur-equivalent then we have $\sum\limits_{i=1}^{m+1} g^P_x(i,m)\cdot F_R(i) = \sum\limits_{i=1}^{m+1} g^Q_x(i,m)\cdot F_R(i)$ for all posets $R$. For any $m$, consider applying this equation to any set of posets $S_1,\ldots,S_m$ such that $|P_i|=i$. Let $g^P_x(i,m) - g^Q_y(i,m) =  c(i,m)$. This gives the system of equations
\begin{align*}
\begin{bmatrix}
F_{S_1}(1) & \cdots & F_{S_1}(n) \\
\vdots & \ddots & \vdots \\
F_{S_n}(1) & \cdots & F_{S_n}(n)
\end{bmatrix}
\begin{bmatrix}
c(1,m)\\
\vdots \\
c(m,m)
\end{bmatrix}
=0
\end{align*}
This matrix is invertible due to the fact that the $F_{S_i}$ are linearly independent. Thus the $c(i,m)$ are 0, for $n \leq m$ $g^P_x(n,m)=g^Q_y(n,m)$, and by definition for $n > m$ both values are 0.
\end{proof}
\begin{corollary}
\label{urequivalenceresult}
For $x \in P$ and $y \in Q$ with $|P|=|Q|=n$, $x$ and $y$ are Ur-equivalent if and only if there exist posets $\{S_1,\cdots, S_n\}$ with $|S_i|=i$ such that $P[x\to S_i]\sim Q[y\to S_i], \ \forall \ i \in [n]$
\end{corollary}
Unfortunately, while $g_x^P$ reveals the structure behind Ur-equivalence, in general it is too difficult to calculate to be of practical use. However, one may note that $g_x^P$ is totally determined by the structure of $P[x \to \varnothing]$ and its relation to $P$. Let $A_k$ be the anti-chain of size $k$--a set with no order relations. The structure of $g_x^P$ suggests that we may be able to strengthen the Corollary \ref{urequivalenceresult}:
\begin{conjecture}
\label{urequivalenceconjecture}
For $x \in P$ and $y \in Q$, $x$ and $y$ are Ur-equivalent if and only if $P \sim Q$ and $P[x\to \varnothing]\sim Q[y\to \varnothing]$.
\end{conjecture}
While this conjecture may seem unlikely with the above information alone, like the order polynomial, $g_x^P$ has significant extra structure that is not understood. In fact, the conjecture holds for small posets, and further is equivalent to a number of simpler statements. For instance,
\begin{proposition}
\label{urequivalenceinduction}
If for all posets $P,Q$ where $P \sim Q$, and for some $x\in P$ and $y\in Q$, $P[x\to \varnothing]\sim Q[y\to \varnothing]$, $P[x\to A_2]\sim Q[y\to A_2]$, then Conjecture~\ref{urequivalenceconjecture} holds. 
\end{proposition}
\begin{proof}
We will prove that $P[x\to A_k]\sim Q[y\to A_k]$ for $\forall \ k \in \mathbb{N}$ by induction on $k$. The RHS of Conjecture~\ref{urequivalenceconjecture} is exactly the base case of our induction. Assume $P[x\to A_i]\sim Q[y\to A_i]$ for $0 \leq i \leq k$, and let $P'=P[x \to A_k] \sim Q' = Q[y \to A_k]$. $P'[x \to \varnothing] = P[x \to A_{k-1}] \sim Q[y \to A_{k-1}] = Q'[x \to \varnothing]$. Then our initial assumption gives that $P'[x \to A_2] = P'[x \to A_{k+1}] \sim Q'[y \to A_2] = Q[y \to A_{k+1}]$. This concludes our induction and the result then follows from Corollary~\ref{urequivalenceresult}.
\end{proof}
While we do not know whether the assumption in Proposition~\ref{urequivalenceinduction} is true for all posets $P,Q$, we do know it is true for certain families of posets, such as $C_k$ or $A_k$, and it gives a simpler formulation of Conjecture~\ref{urequivalenceconjecture}.
\section{Posets of Bounded Height}
Due to the computational complexity of the order polynomial, a general classification of doppelgangers seems hopeless. However, there are certain large families for which the order polynomial is computable in polynomial time. For instance, Faigle and Schrader showed that the order polynomial of $P \in \mathscr W_k$, the set $\{P \in \mathscr P_n \ | \ w(P) \leq k\}$ may be computed in $O(n^{2k+1})$. While this set does not have a rigid enough structure to permit classification, a special subset does. Consider $\mathscr H_k$ $\subset \mathscr W_k$, the set $\{P \in \mathscr P_n \ | \ 
= n-k\}$. We will leverage invariants on doppelgangers and the rigid structure of $\mathscr H_k$ to prove that one may reduce classification of doppelgangers to a number of diophantine equations in time dependent on k. In addition, we show that for constant k the order polynomial of posets in this class has time complexity $O(n)$, and is computable in polynomial time for $k=O(log(n)/log(log(n)))$.
\subsection{Invariants}
Proposition \ref{basisformula} introduces an important restriction on the roots of the order polynomial, first shown by Stanley \cite{Order}. Recall that the height of $P$, $h(P)$, is the cardinality of the largest total ordering contained in $P$.
\begin{corollary}
\label{rootresult}
For all posets $P$, $F_P(x)$ vanishes at $x=0,-1,\ldots,-h(P)+1$ but not at $-h(P),-h(P)-1,\ldots$.
\end{corollary}
In particular, doppelganger posets have the same height.
Such invariants that can be easily calculated allow for classification of doppelgangers of certain families of posets. Lemma~\ref{invariants} presents four such invariants that have simple recursive formulas over the operations of disjoint union and ordinal sum. One of these invariants is $e(P)$, the number of linear extensions of $P$. A \textit{linear extension} of a poset $P$ is an order preserving bijection $P\to[|P|]$.
\begin{lemma}
\label{invariants}
If $P\sim Q$ then $|P|=|Q|$, $F_P(2)=F_Q(2)$, $h(P)=h(Q)$, $e(P)=e(Q)$.
Additionally,
\begin{align}
|P+Q|&=|P|+|Q|\\
|P\oplus Q|&=|P|+|Q|\\
F_{P+Q}(2)&=F_P(2)F_Q(2)\\
F_{P\oplus Q}(2)&=F_P(2)+F_Q(2)-1\\
h(P+Q)&=\max(h(P),h(Q))\\
h(P\oplus Q)&=h(P)+h(Q)\\
e(P+Q)&=\binom{|P|+|Q|}{|P|}e(P)e(Q)\\
e(P\oplus Q)&=e(P)e(Q)
\end{align}
for all posets $P,Q$.
\end{lemma}
\begin{proof}
By Proposition~\ref{basisformula}, $|P|=\deg F_P$, $h(P)$ is the index of the first nonzero term of $F_P$ in the $\binom{x+k-1}{k}$ basis, and $e(P)$ is $(\deg F_P)!$ times the leading coefficient of $F_P$ (since $c_{|P|}=e(P)$ in the notation of Proposition~\ref{basisformula}).
Then all four invariants depend only on $F_P$ which shows the first part of the lemma.
The recursive formulas follow from elementary combinatorial arguments.
\end{proof}
Corollary~\ref{rootresult} states that $F_P$ has roots at $0,-1,\ldots,-h(P)+1$.
This allows for the following necessary and sufficient condition for two posets to be doppelgangers.
\begin{proposition}
\label{prop:classification}
$P\sim Q$ if and only if $|P|=|Q|$, $h(P)=h(Q)$, $e(P)=e(Q)$, and $F_P$ and $F_Q$ agree at $|P|-h(P)-1$ distinct points (not counting the trivial agreement at $1,0,-1,\ldots,-h(P)$).
\end{proposition}
\begin{proof}
The first statement of Lemma~\ref{invariants} shows the forward direction.
For the converse, note that $F_P$ and $F_Q$ have the same leading coefficient of $e(P)/|P|!$.
Then subtracting $F_P$ and $F_Q$ results in a polynomial of degree at most $|P|-1$ which vanishes at $|P|$ points (namely, at $x=1,0,-1,\ldots,-h(P)+1$ and the $|P|-h(P)-1$ other points).
Then $F_P-F_Q$ is identically zero and $P\sim Q$.
\end{proof}
\begin{corollary}
\label{heightpm1}
If $h(P)=|P|-1$ then $P\sim Q$ if and only if $|P|=|Q|$, $h(P)=h(Q)$, and $e(P)=e(Q)$.
\end{corollary}
\begin{corollary}
\label{heightpm2}
If $h(P)=|P|-2$ then $P\sim Q$ iff $|P|=|Q|$, $F_P(2)=F_Q(2)$, $h(P)=h(Q)$, and $e(P)=e(Q)$.
\end{corollary}
In the section that follows we will not use the $e(P)$ invariant, but it is particularly useful for enumerating $\mathscr H_1$ and $\mathscr H_2$.
Another useful invariant is the value $(-1)^{|P|}F_P(-h(P))$ which is equal to 1 if and only if every element of $P$ is contained in a chain of cardinality $h(P)$.
A proof of this result can be found in the appendix.
\subsection{Classifying $\mathscr H_k$}
The height invariant, Corollary \ref{rootresult}, and the underlying structure of $\mathscr H_k$ allow us to theoretically classify all its doppelgangers in time dependent on $k$. 
In addition, leveraging this same structure allows us to efficiently compute the order polynomial of posets in $\mathscr H_k$ in time $O(n)$.
\begin{lemma}
\label{chainstructure}
If $x_1\leq\cdots\leq x_h$ is a chain in $P$ and $x$ is some other element of $P$, then there exist nonnegative integers $a+b+c=h$ such that $x$ is greater than $x_1,\cdots,x_a$, $x$ is incomparable to $x_{a+1},\cdots,x_{a+b}$, and $x$ is less than $x_{a+b+1},\cdots,x_{a+b+c}$.
\end{lemma}
\begin{proof}
Let $m_1$ be maximal such that $x_{m_1}\leq x$, let $m_2$ be minimal such that $x\leq x_{m_1+m_2+1}$, and let $m_3=n-m_1-m_2$.
Then by transitivity, $x$ is greater than $x_1,\cdots,x_{m_1}$ and $x$ is less than $x_{m_1+m_2+1},\cdots,x_{m_1+m_2+m_3}$.
Additionally, $x$ is neither less than nor greater than $x_{m_1+1},\cdots,x_{m_1+m_2}$.
\end{proof}
\begin{lemma}
\label{lem:F_P}
Let $P$ be a finite poset consisting of a chain $x_1\leq\ldots\leq x_{h(P)}$ and $k=|P|-h(P)$ other elements off the chain $y_1,\ldots,y_k$. Consider applying Lemma \ref{chainstructure} to each $y_i$, resulting in values $a$ and $a+b$ for each term. For convenience, we define $(a_1 \leq \ldots \leq a_{2k})$ to be the ordering of these $2k$ values, and further define $a_0=0\leq a_1$, and $a_{2k+1} = h(P)+1 \geq a_{2k}$. Let $d_i, 0 \leq i \leq 2k$, be the difference between the $i$ and $i+1$st terms in this sequence, i.e. $d_i = a_{i+1}-a_{i}$. The value of $F_P(m)$ is a polynomial in the $d_i$ and can be computed in $O(m^{3k+1})$.
\end{lemma}
\begin{proof}
To count the number of order preserving $f\colon P\to[m]$, we sum over at most $m^k$ possible choices of the values of $f$ on the $y_i$.
Note that the values of $f$ on the $x_i$ are completely determined by the locations of the $m-1$ locations where the value of $f$ increases (note that these increases may occur before $x_1$ or after $x_{h(P)}$).
Then for each choice of the value of $f$ on the $y_i$, we sum over the possible choices for how many times $f$ increases between each $x_{a_i}$ and $x_{a_i+1}$.
There are $2k+1$ such pairs and $m-1$ increases so a stars and bars argument gives that there are $\binom{2k+m}{2k+1}$ possible choices for how many times $f$ increases between each pair of consecutive $a_i$.
Then we are summing over at most $m^k\binom{2k+m}{2k+1}$ ways to choose the values of $f$ on the $y_i$ and the locations of the increases of $f$ on the chain, relative to the $a_i$.
Finally, each summand will be a product of $\binom{d_i+j-1}{j}$, where $j$ is the number of increases in between $x_{a_{i}}$ and $x_{a_{i+1}}$. Thus by symmetry there are a total of $m^k\binom{2k+m}{2k+1}$ steps, and as $\binom{2k+m}{2k+1}$ is a polynomial in $m$ of degree at most $2k+1$, this is $O(m^{3k+1})$.
\end{proof}
\begin{lemma}
\label{lem:F_P2}
Given a poset $P$, $|P|=n$, with $h(P)=n-k$, computing the $a_i$ and $d_i$ of Lemma \ref{lem:F_P} takes $O(n)$ time.
\end{lemma}
\begin{proof}
Once we have identified our maximal chain, it is easy to compute $a_i$ and $d_i$ in linear time. Thus, we first prove that we may identify this chain of $P$ in linear time. We require that $P$ be given as a Hasse Diagram, a directed acyclic graph ($DAG$) $G(V,E)$ of cover relations. Given our restriction $h(P) = n-k$, we first wish to bound $|E|$. For analysis, we will partition $V$ into the sets $C$, $n-k$ nodes on our maximal chain, and $Q$, the $k$ ``off-chain" nodes, and count the edges within and between $C$ and $Q$. $C$ is our chain, and thus has exactly $n-k-1$ internal edges. $Q$ could be any poset of size $k$, but because $G$ is a $DAG$, there can be at most $\frac{k(k-1)}{2}$ internal edges. Finally, by the structure outlined in Lemma \ref{lem:F_P}, there can be at most $2k$ edges between $C$ and $Q$. Together, these give the bound
\begin{align*}
|E|\leq n + \frac{(k+1)k}{2}.
\end{align*}
This allows us to run Depth First Search (DFS) on $G$ in $O(n)$ time. Treating $G$ as an undirected graph, we may find all local minima and maxima of $P$ (sinks of $G$) by running a DFS from any node along each connected component. Because $h(P)=n-k$, there can be at most $k+1$ local minima and $k+1$ local maxima. From here finding $C$ is a simple matter of finding the longest path between any minima and maxima, which has a well known linear solution in $O(|V| + |E|)=O(n)$. The complement of $C$ gives $Q$, and the at most $2k$ edges between $C$ and $Q$ recover the $a_i$ and $d_i$ of Lemma \ref{lem:F_P} in linear time.
\end{proof}
We are now prepared to prove Theorem \ref{class-time}, which we split into three components.
\begin{theorem}
\label{prop:complexity}
For $|P|=n$ and constant $k$, the order polynomial of $P \in \mathscr H_k$ can be computed in $O(n)$ time.
\end{theorem}
\begin{proof}
We claim that given all $a_i$ and $d_i$, computing the order polynomial is poly-logarithmic.
Lemma \ref{rootresult} allows us to compute the factored form of the order polynomial by polynomial interpolation and factorization of the remaining k roots. Let this polynomial be $f_k(x)$, then 
\[
F_P(x)=f_k(x)\prod\limits_{i=0}^{n-(k+1)}(x+i)
\]
We may compute $f_k(x)$ by polynomial interpolation on $F_P(1)=1,F_P(2),\ldots,F_P(k+1)$. With these values in hand, interpolation and factorization is polynomial in k by the ``LLL" algorithm \cite{LLL}. Assuming we know the structure of $P$, Lemma $\ref{lem:F_P}$ shows each $F_P(i)$ as a polynomial in the gaps between adjacent $a_i$ may be computed in $O(i^{3k+1})$. 
Given these values, evaluating the polynomial requires summing $O(i^{3k+1})$ products, each with $i-1$ terms of at most $log(n)$ bits. For simplicity, let $n^2$ be the complexity of n-bit multiplication, then computing the products takes $O(i^{3k+1}i^2log^2(n))$. The evaluated products are of at most $ilog(n)$ bits. Summing these then takes $i^{3k+1}(i^{3k}+ilog(n))$, so computing and evaluating all $F_P(i)$ takes poly-log time. Thus Lemma \ref{lem:F_P2} shows that computing the $a_i$ and $d_i$ is our bottleneck, and computing the order polynomial takes $O(n)$ time. 
\end{proof}
For constant k, we have shown that our subfamily of Faigle and Schrader's $\mathscr W_k$ may be computed in $O(n)$ time, and thus does not grow asymptotically in k as their $O(n^{2k+1})$ bound does. This allows us to extend our family to non-constant values of $k$.
\begin{corollary}
For $|P|=n$ and $k=O(\frac{log(n)}{log(log(n))})$, the order polynomial of $P \in \mathscr H_k$ may be computed in polynomial time.
\end{corollary}
\begin{proof}
Proposition \ref{prop:complexity} showed that we can compute the order polynomial in $O(n+k^{3k+2}(k^{3k} + k^2log^2(n)))$ time. Setting $k=c\frac{log(n)}{log(log(n))}$ gives $k^k=O(n^c)$, and thus the leading term becomes $k^{6k+2}$ which is polynomial for $k=O(\frac{log(n)}{log(log(n))})$
\end{proof}
\begin{proposition}
The doppelgangers of posets in $\mathscr H_k$ can be completely classified up to sets of k diophantine equations in $2^{O(k^2)}$ time.
\end{proposition}
\begin{proof}
First note that by Lemma \ref{invariants}, doppelgangers of posets in $\mathscr H_k$ will themselves be posets in $\mathscr H_k$ of the same height and size.
Thus, it suffices to classify doppelgangers within $\mathscr H_k$.
If $P\in\mathscr H_k$ then $P$ consists of a chain $C_{h(P)}=\{x_1\leq\ldots\leq x_{h(P)}\}$ and $k$ other elements off the chain $Q=\{y_1,\ldots,y_k\}$.
For each $y_i$, let $a_i$ denote the number of $x_j\leq y_i$ and let $b_i$ denote the number of $x_j\not\geq y_i$.
We call a consistent choice of both a $Q$ and a relative ordering between the $a_i$ and the $b_i$ a \textit{family in }$\mathscr H_k$. Enumerating all choices of $Q$ simply corresponds to enumerating all posets of size k, which takes $2^{O(k^2)}$ time. There are at most $(2k)!$ relative orderings between the $a_i$ and $b_i$, and given a choice of $Q$, checking any given ordering is consistent takes time polynomial in $k$. For each family in $\mathscr H_k$, Lemma \ref{lem:F_P} shows the values of $F_P(2),\ldots,F_P(k+1)$ can be written as a polynomial function of the distances between the $a_i$ and $b_i$ in $O(k^{3k+1})$ time.
By Corollary \ref{rootresult}, if $P,Q$ are posets with $|P|=|Q|=h(P)+k=h(Q)+k$ then $P\sim Q$ if and only if $F_P(i)=F_Q(i)$ for $i=2,\ldots,k+1$.
Thus, the doppelgangers of posets in $\mathscr H_k$ can be completely classified up to sets of k diophantine equations in $2^{O(k^2)}$ time.
\end{proof}
\subsection{Example: $\mathscr H_1$ and $\mathscr H_2$}
While for large k, classifying the $\mathscr H_k$ may be computationally intractable, $\mathscr H_1$ and $\mathscr H_2$ are simple enough to compute by hand. We provide a classification of these families as an example of the above method, and show how the diophantine equations lead to new infinite families of doppelgangers. Note however that we choose to use the number of linear extensions $e(P)$ rather than $F_P(3)$, as described in Corollaries \ref{heightpm1} and \ref{heightpm2}. We begin by enumerating the families of $\mathscr H_1$ and $\mathscr H_2$
\begin{proposition}
All posets $P$ with $|P|-h(P)=1$ are isomorphic to a poset depicted by Figure~\ref{fig:h2}(a).
\end{proposition}
\begin{proof}
Let $C$ be a maximal chain in $P$ and let $x$ be the remaining element of $P$.
Let $m_1,m_2,m_3$ be the result of applying Lemma~\ref{chainstructure} to $C$ and $x$.
Then $P\cong Tri(m_1,m_2,m_3)$.
\end{proof}
\begin{proposition}
All posets $P$ with $|P|-h(P)=2$ are isomorphic to poset depicted by Figures~\ref{fig:h2}(b-e).
\end{proposition}
\begin{proof}
Let $C$ be a maximal chain in $P$ and let $x,y$ be the two remaining elements of $P$.
Let $m_1,m_2,m_3$ be the result of applying Lemma~\ref{chainstructure} to $C$ and $x$ and let $n_1,n_2,n_3$ be the result of applying Lemma~\ref{chainstructure} to $C$ and $y$.
Then
\[P\cong\begin{cases}
Ntri(m_1,n_1-m_1,n_2,n_3-m_3,m_3)
&m_1\leq n_1,\ m_3\leq n_3,\ x\text{ and }y\text{ incomparable}
\\Ntri(n_1,m_1-n_1,m_2,m_3-n_3,n_3)
&m_1\geq n_1,m_3\geq n_3,\ x\text{ and }y\text{ incomparable}
\\Xdis(m_1,n_1-m_1,m_1+m_2-n_1,m_3-n_3,n_3)
&m_1\leq n_1,m_3\geq n_3,\ x\text{ and }y\text{ incomparable}
\\Xdis(n_1,m_1-n_1,n_1+n_2-m_1,n_3-m_3,m_3)
&m_1\geq n_1,m_3\leq n_3,\ x\text{ and }y\text{ incomparable}
\\Xcon(m_1,n_1-m_1,m_1+m_2-n_1,m_3-n_3,n_3)
&m_1+m_2-n_1\geq0,\ x\leq y
\\Xcon(n_1,m_1-n_1,n_1+n_2-m_1,n_3-m_3,m_3)
&n_1+n_2-m_1\geq0,\ y\leq x
\\Dtri(m_1,m_2,n_1-m_1-m_2,n_2,n_3)
&n_1-m_1-m_2\geq0,\ x\leq y
\\Dtri(n_1,n_2,m_1-n_1-n_2,m_2,m_3)
&m_1-n_1-n_2\geq0,\ y\leq x
\end{cases}.\]
\end{proof}
The values of the invariants for the posets in Figure~\ref{fig:h2} are given in Table \ref{tab:dio} and the computation of these values can be found in the appendix.
The result of this table is that we can compute all doppelgangers among posets of height at most $|P|-2$ by solving various pairs of Diophantine equations. All pairs lead to infinite families of doppelgangers such as that depicted in Figure \ref{fig:h2_2}.
\begin{example}
In the below, we drop variables which do not appear in $e(P)$ or $F_P(2)$. For instance, Dtri$(m_1,m_2,m_3,m_4)$ becomes Dtri$(m_2,m_4)$. Further, variables are assumed to be constrained in such a manner that every $m_i$ is $>0$.
\begin{enumerate}
\item Dtri$(a(a+2),a-1)$ $\sim$ Ntri$(a^2 - 2,a)$.
\item Dtri$(4a, a-1)$ $\sim$ Xdis$(a-2,a,a)$.
\item Dtri$(5a-n+1,2a-n)$ $\sim$ Xcon$(a - n,2a,2a-n+1)$. 
\item Ntri$(3a - 5c,2c)$ $\sim$ Xdis$(3c-a,a,a-2c)$.
\item Ntri$(2d-b,b)$ $\sim$ Ntri$(2b-d,d)$.  \label{Ntri2}
\end{enumerate}
\end{example}
\begin{proposition}
Doppelgangers between Dtri and Dtri are fully classified by Corollary~\ref{opluscommute}.
\end{proposition}
\begin{proof}
Consider Dtri$(a,b)$ and Dtri$(c,d)$. The equation for $F_P(2)$ gives $a = c + d - b$. Plugging this into our $e(P)$ equation gives $(b-c)(b-d) = 0$. Thus the only solutions are $b=c$ or $b=d$, both of which are given by Corollary~\ref{opluscommute}.
\end{proof}
\begin{proposition}
Doppelgangers between Ntri and Ntri are fully classified by Corollary~\ref{opluscommute} and Equation~\ref{Ntri2}.
\end{proposition}
\begin{proof}
Consider Ntri$(a,b)$ and Ntri$(c,d)$. The equation for $F_P(2)$ gives $c=a+3(b-d)$. Plugging this into our $e(P)$ equation gives
\[
(a+b+2)(b+1) = (a+3(b-d) + d + 2)(d + 1).
\]
Expanding and re-factoring gives $(b-d)(a+b-2d)=0$. There are two possible cases:
\begin{enumerate}
\item $b=d$: In this case, $c=a+3(b-d)=a$ which is given by Corollary \ref{opluscommute}.
\item $a+b=2d$: In this case $c=a+3(b-d)=2b-d$ and $a=2d-b$ which is given by Equation \ref{Ntri2}.

We have enumerated all possible solutions to the equation.
\end{enumerate}
\end{proof}
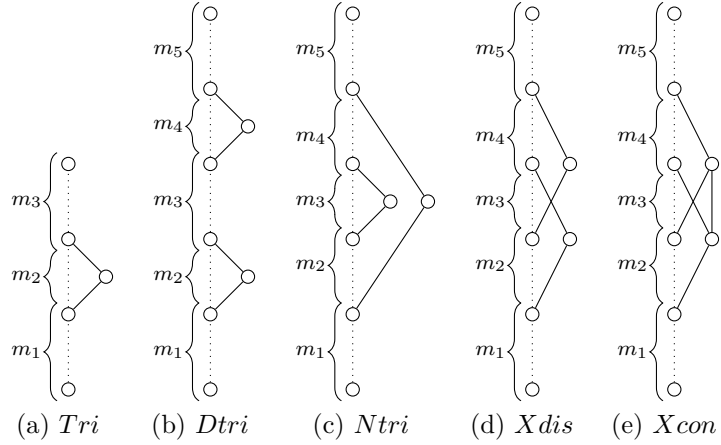
\begin{figure}
\begin{center}
\begin{tabular}{c c c c c}
\begin{tikzpicture}[scale=0.5]
\node[draw, circle, inner sep=0pt, minimum size=5pt] at (0,0) (a) {};
\node[draw, circle, inner sep=0pt, minimum size=5pt] at (0,2) (b) {};
\node[draw, circle, inner sep=0pt, minimum size=5pt] at (0,4) (c) {};
\node[draw, circle, inner sep=0pt, minimum size=5pt] at (0,6) (d) {};
\node[draw, circle, inner sep=0pt, minimum size=5pt] at (1,3) (e) {};
\draw[dotted] (a) -- (b) -- (c) -- (d);
\draw (b) -- (e) -- (c);
\draw [decorate,decoration={brace,amplitude=5pt}]
(-0.3,-0.3) -- (-0.3,2.3) node [black,midway,xshift=-0.4cm]
{\footnotesize $m_1$};
\draw [decorate,decoration={brace,amplitude=5pt}]
(-0.3,2.3) -- (-0.3,3.7) node [black,midway,xshift=-0.4cm]
{\footnotesize $m_2$};
\draw [decorate,decoration={brace,amplitude=5pt}]
(-0.3,3.7) -- (-0.3,6.3) node [black,midway,xshift=-0.4cm]
{\footnotesize $m_3$};
\end{tikzpicture}
&\begin{tikzpicture}[scale=0.5]
\node[draw, circle, inner sep=0pt, minimum size=5pt] at (0,0) (a) {};
\node[draw, circle, inner sep=0pt, minimum size=5pt] at (0,2) (b) {};
\node[draw, circle, inner sep=0pt, minimum size=5pt] at (0,4) (c) {};
\node[draw, circle, inner sep=0pt, minimum size=5pt] at (0,6) (d) {};
\node[draw, circle, inner sep=0pt, minimum size=5pt] at (0,8) (e) {};
\node[draw, circle, inner sep=0pt, minimum size=5pt] at (0,10) (f) {};
\node[draw, circle, inner sep=0pt, minimum size=5pt] at (1,3) (g) {};
\node[draw, circle, inner sep=0pt, minimum size=5pt] at (1,7) (h) {};
\draw[dotted] (a) -- (b) -- (c) -- (d) -- (e) -- (f);
\draw (b) -- (g) -- (c);
\draw (d) -- (h) -- (e);
\draw [decorate,decoration={brace,amplitude=5pt}]
(-0.3,-0.3) -- (-0.3,2.3) node [black,midway,xshift=-0.4cm]
{\footnotesize $m_1$};
\draw [decorate,decoration={brace,amplitude=5pt}]
(-0.3,2.3) -- (-0.3,3.7) node [black,midway,xshift=-0.4cm]
{\footnotesize $m_2$};
\draw [decorate,decoration={brace,amplitude=5pt}]
(-0.3,3.7) -- (-0.3,6.3) node [black,midway,xshift=-0.4cm]
{\footnotesize $m_3$};
\draw [decorate,decoration={brace,amplitude=5pt}]
(-0.3,6.3) -- (-0.3,7.7) node [black,midway,xshift=-0.4cm]
{\footnotesize $m_4$};
\draw [decorate,decoration={brace,amplitude=5pt}]
(-0.3,7.7) -- (-0.3,10.3) node [black,midway,xshift=-0.4cm]
{\footnotesize $m_5$};
\end{tikzpicture}
&\begin{tikzpicture}[scale=0.5]
\node[draw, circle, inner sep=0pt, minimum size=5pt] at (0,0) (a) {};
\node[draw, circle, inner sep=0pt, minimum size=5pt] at (0,2) (b) {};
\node[draw, circle, inner sep=0pt, minimum size=5pt] at (0,4) (c) {};
\node[draw, circle, inner sep=0pt, minimum size=5pt] at (0,6) (d) {};
\node[draw, circle, inner sep=0pt, minimum size=5pt] at (0,8) (e) {};
\node[draw, circle, inner sep=0pt, minimum size=5pt] at (0,10) (f) {};
\node[draw, circle, inner sep=0pt, minimum size=5pt] at (1,5) (g) {};
\node[draw, circle, inner sep=0pt, minimum size=5pt] at (2,5) (h) {};
\draw[dotted] (a) -- (b) -- (c) -- (d) -- (e) -- (f);
\draw (c) -- (g) -- (d);
\draw (b) -- (h) -- (e);
\draw [decorate,decoration={brace,amplitude=5pt}]
(-0.3,-0.3) -- (-0.3,2.3) node [black,midway,xshift=-0.4cm]
{\footnotesize $m_1$};
\draw [decorate,decoration={brace,amplitude=5pt}]
(-0.3,2.3) -- (-0.3,4.3) node [black,midway,xshift=-0.4cm]
{\footnotesize $m_2$};
\draw [decorate,decoration={brace,amplitude=5pt}]
(-0.3,4.3) -- (-0.3,5.7) node [black,midway,xshift=-0.4cm]
{\footnotesize $m_3$};
\draw [decorate,decoration={brace,amplitude=5pt}]
(-0.3,5.7) -- (-0.3,7.7) node [black,midway,xshift=-0.4cm]
{\footnotesize $m_4$};
\draw [decorate,decoration={brace,amplitude=5pt}]
(-0.3,7.7) -- (-0.3,10.3) node [black,midway,xshift=-0.4cm]
{\footnotesize $m_5$};
\end{tikzpicture}
&\begin{tikzpicture}[scale=0.5]
\node[draw, circle, inner sep=0pt, minimum size=5pt] at (0,0) (a) {};
\node[draw, circle, inner sep=0pt, minimum size=5pt] at (0,2) (b) {};
\node[draw, circle, inner sep=0pt, minimum size=5pt] at (0,4) (c) {};
\node[draw, circle, inner sep=0pt, minimum size=5pt] at (0,6) (d) {};
\node[draw, circle, inner sep=0pt, minimum size=5pt] at (0,8) (e) {};
\node[draw, circle, inner sep=0pt, minimum size=5pt] at (0,10) (f) {};
\node[draw, circle, inner sep=0pt, minimum size=5pt] at (1,4) (g) {};
\node[draw, circle, inner sep=0pt, minimum size=5pt] at (1,6) (h) {};
\draw[dotted] (a) -- (b) -- (c) -- (d) -- (e) -- (f);
\draw (b) -- (g) -- (d);
\draw (c) -- (h) -- (e);
\draw [decorate,decoration={brace,amplitude=5pt}]
(-0.3,-0.3) -- (-0.3,2.3) node [black,midway,xshift=-0.4cm]
{\footnotesize $m_1$};
\draw [decorate,decoration={brace,amplitude=5pt}]
(-0.3,2.3) -- (-0.3,4.3) node [black,midway,xshift=-0.4cm]
{\footnotesize $m_2$};
\draw [decorate,decoration={brace,amplitude=5pt}]
(-0.3,4.3) -- (-0.3,5.7) node [black,midway,xshift=-0.4cm]
{\footnotesize $m_3$};
\draw [decorate,decoration={brace,amplitude=5pt}]
(-0.3,5.7) -- (-0.3,7.7) node [black,midway,xshift=-0.4cm]
{\footnotesize $m_4$};
\draw [decorate,decoration={brace,amplitude=5pt}]
(-0.3,7.7) -- (-0.3,10.3) node [black,midway,xshift=-0.4cm]
{\footnotesize $m_5$};
\end{tikzpicture}
&\begin{tikzpicture}[scale=0.5]
\node[draw, circle, inner sep=0pt, minimum size=5pt] at (0,0) (a) {};
\node[draw, circle, inner sep=0pt, minimum size=5pt] at (0,2) (b) {};
\node[draw, circle, inner sep=0pt, minimum size=5pt] at (0,4) (c) {};
\node[draw, circle, inner sep=0pt, minimum size=5pt] at (0,6) (d) {};
\node[draw, circle, inner sep=0pt, minimum size=5pt] at (0,8) (e) {};
\node[draw, circle, inner sep=0pt, minimum size=5pt] at (0,10) (f) {};
\node[draw, circle, inner sep=0pt, minimum size=5pt] at (1,4) (g) {};
\node[draw, circle, inner sep=0pt, minimum size=5pt] at (1,6) (h) {};
\draw[dotted] (a) -- (b) -- (c) -- (d) -- (e) -- (f);
\draw (b) -- (g) -- (d);
\draw (c) -- (h) -- (e);
\draw (g) -- (h);
\draw [decorate,decoration={brace,amplitude=5pt}]
(-0.3,-0.3) -- (-0.3,2.3) node [black,midway,xshift=-0.4cm]
{\footnotesize $m_1$};
\draw [decorate,decoration={brace,amplitude=5pt}]
(-0.3,2.3) -- (-0.3,4.3) node [black,midway,xshift=-0.4cm]
{\footnotesize $m_2$};
\draw [decorate,decoration={brace,amplitude=5pt}]
(-0.3,4.3) -- (-0.3,5.7) node [black,midway,xshift=-0.4cm]
{\footnotesize $m_3$};
\draw [decorate,decoration={brace,amplitude=5pt}]
(-0.3,5.7) -- (-0.3,7.7) node [black,midway,xshift=-0.4cm]
{\footnotesize $m_4$};
\draw [decorate,decoration={brace,amplitude=5pt}]
(-0.3,7.7) -- (-0.3,10.3) node [black,midway,xshift=-0.4cm]
{\footnotesize $m_5$};
\end{tikzpicture}
\\(a) $Tri$
&(b) $Dtri$
&(c) $Ntri$
&(d) $Xdis$
&(e) $Xcon$
\end{tabular}
\end{center}
\caption{Five infinite families of posets.
\label{fig:h2}}
\end{figure}
\begin{table}
\centering%
\begin{tabular}{c|c|c}
Infinite Family&$e(P)$&$F_P(2)$\\
$Tri(m_2)$&$m_2+1$&$|P|+m_2+1$\\
$Dtri(m_2,m_4)$&$(m_2+1)(m_4+1)$&$|P|+m_2+m_4+1$\\
$Ntri(m_2+m_4,m_3)$&$(m_2+m_3+m_4+2)(m_3+1)$&$|P|+m_2+3m_3+m_4+2$\\
$Xdis(m_2,m_3,m_4)$&$(m_2+m_3+1)(m_3+m_4+1)+m_3+1$&$|P|+m_2+3m_3+m_4+2$\\
$Xcon(m_2,m_3,m_4)$&$(m_2+m_3+1)(m_3+m_4+1)-\frac{1}{2}m_3(m_3+1)$&$|P|+m_2+2m_3+m_4+1$
\end{tabular}
\caption{Values of $e(P)$ and $F_P(2)$ for the five infinite families in Figure 1.}
\label{tab:dio}
\end{table}
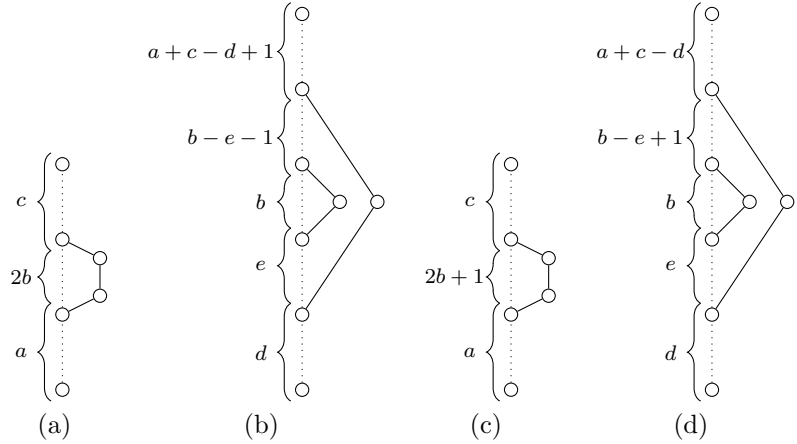
\begin{figure}[!htb]
\begin{center}
\begin{tabular}{c c c c c c}
\begin{tikzpicture}[scale=0.5]
\node[draw, circle, inner sep=0pt, minimum size=5pt] at (0,2) (b) {};
\node[draw, circle, inner sep=0pt, minimum size=5pt] at (0,4) (c) {};
\node[draw, circle, inner sep=0pt, minimum size=5pt] at (0,6) (d) {};
\node[draw, circle, inner sep=0pt, minimum size=5pt] at (0,8) (e) {};
\node[draw, circle, inner sep=0pt, minimum size=5pt] at (1,4.5) (g) {};
\node[draw, circle, inner sep=0pt, minimum size=5pt] at (1,5.5) (h) {};
\draw[dotted] (b) -- (c) -- (d) -- (e);
\draw (c) -- (g) -- (h) -- (d);
\draw [decorate,decoration={brace,amplitude=5pt}]
(-0.3,1.7) -- (-0.3,4.3) node [black,midway,xshift=-0.4cm]
{\footnotesize $a$};
\draw [decorate,decoration={brace,amplitude=5pt}]
(-0.3,4.3) -- (-0.3,5.7) node [black,midway,xshift=-0.4cm]
{\footnotesize $2b$};
\draw [decorate,decoration={brace,amplitude=5pt}]
(-0.3,5.7) -- (-0.3,8.3) node [black,midway,xshift=-0.4cm]
{\footnotesize $c$};
\end{tikzpicture}
&\begin{tikzpicture}[scale=0.5]
\node[draw, circle, inner sep=0pt, minimum size=5pt] at (0,0) (a) {};
\node[draw, circle, inner sep=0pt, minimum size=5pt] at (0,2) (b) {};
\node[draw, circle, inner sep=0pt, minimum size=5pt] at (0,4) (c) {};
\node[draw, circle, inner sep=0pt, minimum size=5pt] at (0,6) (d) {};
\node[draw, circle, inner sep=0pt, minimum size=5pt] at (0,8) (e) {};
\node[draw, circle, inner sep=0pt, minimum size=5pt] at (0,10) (f) {};
\node[draw, circle, inner sep=0pt, minimum size=5pt] at (1,5) (g) {};
\node[draw, circle, inner sep=0pt, minimum size=5pt] at (2,5) (h) {};
\draw[dotted] (a) -- (b) -- (c) -- (d) -- (e) -- (f);
\draw (c) -- (g) -- (d);
\draw (b) -- (h) -- (e);
\draw [decorate,decoration={brace,amplitude=5pt}]
(-0.3,-0.3) -- (-0.3,2.3) node [black,midway,xshift=-0.4cm]
{\footnotesize $d$};
\draw [decorate,decoration={brace,amplitude=5pt}]
(-0.3,2.3) -- (-0.3,4.3) node [black,midway,xshift=-0.4cm]
{\footnotesize $e$};
\draw [decorate,decoration={brace,amplitude=5pt}]
(-0.3,4.3) -- (-0.3,5.7) node [black,midway,xshift=-0.4cm]
{\footnotesize $b$};
\draw [decorate,decoration={brace,amplitude=5pt}]
(-0.3,5.7) -- (-0.3,7.7) node [black,midway,xshift=-0.8cm]
{\footnotesize $b-e-1$};
\draw [decorate,decoration={brace,amplitude=5pt}]
(-0.3,7.7) -- (-0.3,10.3) node [black,midway,xshift=-1.06cm]
{\footnotesize $a+c-d+1$};
\end{tikzpicture}
&\begin{tikzpicture}[scale=0.5]
\node[draw, circle, inner sep=0pt, minimum size=5pt] at (0,2) (b) {};
\node[draw, circle, inner sep=0pt, minimum size=5pt] at (0,4) (c) {};
\node[draw, circle, inner sep=0pt, minimum size=5pt] at (0,6) (d) {};
\node[draw, circle, inner sep=0pt, minimum size=5pt] at (0,8) (e) {};
\node[draw, circle, inner sep=0pt, minimum size=5pt] at (1,4.5) (g) {};
\node[draw, circle, inner sep=0pt, minimum size=5pt] at (1,5.5) (h) {};
\draw[dotted] (b) -- (c) -- (d) -- (e);
\draw (c) -- (g) -- (h) -- (d);
\draw [decorate,decoration={brace,amplitude=5pt}]
(-0.3,1.7) -- (-0.3,4.3) node [black,midway,xshift=-0.4cm]
{\footnotesize $a$};
\draw [decorate,decoration={brace,amplitude=5pt}]
(-0.3,4.3) -- (-0.3,5.7) node [black,midway,xshift=-0.6cm]
{\footnotesize $2b + 1$};
\draw [decorate,decoration={brace,amplitude=5pt}]
(-0.3,5.7) -- (-0.3,8.3) node [black,midway,xshift=-0.4cm]
{\footnotesize $c$};
\end{tikzpicture}
&\begin{tikzpicture}[scale=0.5]
\node[draw, circle, inner sep=0pt, minimum size=5pt] at (0,0) (a) {};
\node[draw, circle, inner sep=0pt, minimum size=5pt] at (0,2) (b) {};
\node[draw, circle, inner sep=0pt, minimum size=5pt] at (0,4) (c) {};
\node[draw, circle, inner sep=0pt, minimum size=5pt] at (0,6) (d) {};
\node[draw, circle, inner sep=0pt, minimum size=5pt] at (0,8) (e) {};
\node[draw, circle, inner sep=0pt, minimum size=5pt] at (0,10) (f) {};
\node[draw, circle, inner sep=0pt, minimum size=5pt] at (1,5) (g) {};
\node[draw, circle, inner sep=0pt, minimum size=5pt] at (2,5) (h) {};
\draw[dotted] (a) -- (b) -- (c) -- (d) -- (e) -- (f);
\draw (c) -- (g) -- (d);
\draw (b) -- (h) -- (e);
\draw [decorate,decoration={brace,amplitude=5pt}]
(-0.3,-0.3) -- (-0.3,2.3) node [black,midway,xshift=-0.4cm]
{\footnotesize $d$};
\draw [decorate,decoration={brace,amplitude=5pt}]
(-0.3,2.3) -- (-0.3,4.3) node [black,midway,xshift=-0.4cm]
{\footnotesize $e$};
\draw [decorate,decoration={brace,amplitude=5pt}]
(-0.3,4.3) -- (-0.3,5.7) node [black,midway,xshift=-0.4cm]
{\footnotesize $b$};
\draw [decorate,decoration={brace,amplitude=5pt}]
(-0.3,5.7) -- (-0.3,7.7) node [black,midway,xshift=-0.8cm]
{\footnotesize $b-e+1$};
\draw [decorate,decoration={brace,amplitude=5pt}]
(-0.3,7.7) -- (-0.3,10.3) node [black,midway,xshift=-0.8cm]
{\footnotesize $a+c-d$};
\end{tikzpicture}
\\(a)
&(b)
&(c)
&(d)
\end{tabular}
\end{center}
\caption{Two Infinite Families of Doppelgangers which follow from Table 1: (a)$\sim$(b) and (c)$\sim$(d)}
\label{fig:h2_2}
\end{figure}
\section{Further directions}
\subsection{The Multivariate Generating Function}
Perhaps the most obvious extension of this paper would be to more carefully investigate the implications of both the proper and improper recurrence relations on the multivariate generating function. Induction on incomparable elements led to a number of nice results over order polynomials, and such a tool could potentially be used to approach some of the questions offered in the end of McNamara and Ward's paper \cite{McNamara}.
\subsection{Single Step Chain Decomposition}
We briefly touched on how our recurrences may be useful, even in only a single application, both in explaining one of McNamara and Ward's posets as well as easily constructing an infinite family of doppelgangers for $C_n + C_n$. What other dopplegangers can be explained through a single set of chain decomposition? $C_n + C_n$ has high structural symmetry and simplicity. In a similar vein, Stanley suggested classifying doppelgangers which cannot be shown in a single step of a recurrence.
\subsection{Closed Families}
Our analysis of $\mathscr H_k$ initially stemmed from the fact that the family is closed under Johnson's recurrence. Further, series-parallel posets are closed under recurrence as well (for the correct choice of incomparable elements). This closure allows for easy recursive calculation of important invariants, and it is no coincidence that our results focus on these families. In fact, Faigle and Schrader's $\mathscr W_k$ is also such a closed family. However, the decomposition is complicated, and width is not an invariant on the order polynomial. However, the idea could carry over to the multivariate generating function where width and height do create invariants on naturally labeled posets.
\subsection{Explaining Non-Series Parallel Doppelgangers with the Ur-Decomposition}
Since the Ur-Decomposition generalized the series-parallel decomposition and exists for all posets, one would hope that the Ur-Decomposition could be used to prove some of the non-series parallel doppelgangers produced in \cite{Hamaker}.
However, many of the posets considered in \cite{Hamaker} are grid-like and do not decompose well under the Ur-Decomposition.
For example, the grid poset $C_n\times C_m$ decomposes as $C_1\oplus P\oplus C_1$ where $P$ is indecomposable under the Ur-operation (this can be seen by noting that every RAP, Definition \ref{RAP}, is a singleton).
\section{Appendix}
\subsection{Computation of Invariants for Posets of large height}
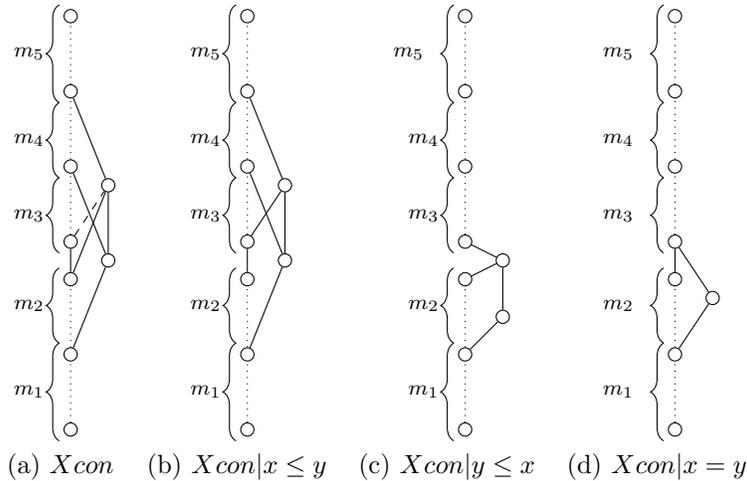
\begin{figure}[htb]
\begin{center}
\begin{tabular}{c c c c}
\begin{tikzpicture}[scale=0.5]
\node[draw, circle, inner sep=0pt, minimum size=5pt] at (0,0) (a) {};
\node[draw, circle, inner sep=0pt, minimum size=5pt] at (0,2) (b) {};
\node[draw, circle, inner sep=0pt, minimum size=5pt] at (0,4) (c) {};
\node[draw, circle, inner sep=0pt, minimum size=5pt] at (0,5) (d) {};
\node[draw, circle, inner sep=0pt, minimum size=5pt] at (0,7) (e) {};
\node[draw, circle, inner sep=0pt, minimum size=5pt] at (0,9) (f) {};
\node[draw, circle, inner sep=0pt, minimum size=5pt] at (0,11) (g) {};
\node[draw, circle, inner sep=0pt, minimum size=5pt] at (1,4.5) (h) {};
\node[draw, circle, inner sep=0pt, minimum size=5pt] at (1,6.5) (i) {};
\draw[dotted] (a) -- (b) -- (c);
\draw (c) -- (d);
\draw[dotted] (d) -- (e) -- (f) -- (g);
\draw[dashed] (d) -- (i);
\draw (b) -- (h) -- (e);
\draw (c) -- (i) -- (f);
\draw (h) -- (i);
\draw [decorate,decoration={brace,amplitude=5pt}]
(-0.3,-0.3) -- (-0.3,2.3) node [black,midway,xshift=-0.4cm]
{\footnotesize $m_1$};
\draw [decorate,decoration={brace,amplitude=5pt}]
(-0.3,2.3) -- (-0.3,4.3) node [black,midway,xshift=-0.4cm]
{\footnotesize $m_2$};
\draw [decorate,decoration={brace,amplitude=5pt}]
(-0.3,4.7) -- (-0.3,6.7) node [black,midway,xshift=-0.4cm]
{\footnotesize $m_3$};
\draw [decorate,decoration={brace,amplitude=5pt}]
(-0.3,6.7) -- (-0.3,8.7) node [black,midway,xshift=-0.4cm]
{\footnotesize $m_4$};
\draw [decorate,decoration={brace,amplitude=5pt}]
(-0.3,8.7) -- (-0.3,11.3) node [black,midway,xshift=-0.4cm]
{\footnotesize $m_5$};
\end{tikzpicture}
&\begin{tikzpicture}[scale=0.5]
\node[draw, circle, inner sep=0pt, minimum size=5pt] at (0,0) (a) {};
\node[draw, circle, inner sep=0pt, minimum size=5pt] at (0,2) (b) {};
\node[draw, circle, inner sep=0pt, minimum size=5pt] at (0,4) (c) {};
\node[draw, circle, inner sep=0pt, minimum size=5pt] at (0,5) (d) {};
\node[draw, circle, inner sep=0pt, minimum size=5pt] at (0,7) (e) {};
\node[draw, circle, inner sep=0pt, minimum size=5pt] at (0,9) (f) {};
\node[draw, circle, inner sep=0pt, minimum size=5pt] at (0,11) (g) {};
\node[draw, circle, inner sep=0pt, minimum size=5pt] at (1,4.5) (h) {};
\node[draw, circle, inner sep=0pt, minimum size=5pt] at (1,6.5) (i) {};
\draw[dotted] (a) -- (b) -- (c);
\draw (c) -- (d) -- (i);
\draw[dotted] (d) -- (e) -- (f) -- (g);
\draw (b) -- (h) -- (e);
\draw (h) -- (i) -- (f);
\draw [decorate,decoration={brace,amplitude=5pt}]
(-0.3,-0.3) -- (-0.3,2.3) node [black,midway,xshift=-0.4cm]
{\footnotesize $m_1$};
\draw [decorate,decoration={brace,amplitude=5pt}]
(-0.3,2.3) -- (-0.3,4.3) node [black,midway,xshift=-0.4cm]
{\footnotesize $m_2$};
\draw [decorate,decoration={brace,amplitude=5pt}]
(-0.3,4.7) -- (-0.3,6.7) node [black,midway,xshift=-0.4cm]
{\footnotesize $m_3$};
\draw [decorate,decoration={brace,amplitude=5pt}]
(-0.3,6.7) -- (-0.3,8.7) node [black,midway,xshift=-0.4cm]
{\footnotesize $m_4$};
\draw [decorate,decoration={brace,amplitude=5pt}]
(-0.3,8.7) -- (-0.3,11.3) node [black,midway,xshift=-0.4cm]
{\footnotesize $m_5$};
\end{tikzpicture}
&\begin{tikzpicture}[scale=0.5]
\node[draw, circle, inner sep=0pt, minimum size=5pt] at (0,0) (a) {};
\node[draw, circle, inner sep=0pt, minimum size=5pt] at (0,2) (b) {};
\node[draw, circle, inner sep=0pt, minimum size=5pt] at (0,4) (c) {};
\node[draw, circle, inner sep=0pt, minimum size=5pt] at (0,5) (d) {};
\node[draw, circle, inner sep=0pt, minimum size=5pt] at (0,7) (e) {};
\node[draw, circle, inner sep=0pt, minimum size=5pt] at (0,9) (f) {};
\node[draw, circle, inner sep=0pt, minimum size=5pt] at (0,11) (g) {};
\node[draw, circle, inner sep=0pt, minimum size=5pt] at (1,3) (h) {};
\node[draw, circle, inner sep=0pt, minimum size=5pt] at (1,4.5) (i) {};
\draw[dotted] (a) -- (b) -- (c);
\draw[dotted] (d) -- (e) -- (f) -- (g);
\draw (d) -- (i);
\draw (b) -- (h) -- (i) -- (c);
\draw [decorate,decoration={brace,amplitude=5pt}]
(-0.3,-0.3) -- (-0.3,2.3) node [black,midway,xshift=-0.4cm]
{\footnotesize $m_1$};
\draw [decorate,decoration={brace,amplitude=5pt}]
(-0.3,2.3) -- (-0.3,4.3) node [black,midway,xshift=-0.4cm]
{\footnotesize $m_2$};
\draw [decorate,decoration={brace,amplitude=5pt}]
(-0.3,4.7) -- (-0.3,6.7) node [black,midway,xshift=-0.4cm]
{\footnotesize $m_3$};
\draw [decorate,decoration={brace,amplitude=5pt}]
(-0.3,6.7) -- (-0.3,8.7) node [black,midway,xshift=-0.4cm]
{\footnotesize $m_4$};
\draw [decorate,decoration={brace,amplitude=5pt}]
(-0.3,8.7) -- (-0.3,11.3) node [black,midway,xshift=-0.6cm]
{\footnotesize $m_5$};
\end{tikzpicture}
&\begin{tikzpicture}[scale=0.5]
\node[draw, circle, inner sep=0pt, minimum size=5pt] at (0,0) (a) {};
\node[draw, circle, inner sep=0pt, minimum size=5pt] at (0,2) (b) {};
\node[draw, circle, inner sep=0pt, minimum size=5pt] at (0,4) (c) {};
\node[draw, circle, inner sep=0pt, minimum size=5pt] at (0,5) (d) {};
\node[draw, circle, inner sep=0pt, minimum size=5pt] at (0,7) (e) {};
\node[draw, circle, inner sep=0pt, minimum size=5pt] at (0,9) (f) {};
\node[draw, circle, inner sep=0pt, minimum size=5pt] at (0,11) (g) {};
\node[draw, circle, inner sep=0pt, minimum size=5pt] at (1,3.5) (h) {};
\draw[dotted] (a) -- (b) -- (c);
\draw (c) -- (d);
\draw[dotted] (d) -- (e) -- (f) -- (g);
\draw (b) -- (h) -- (d);
\draw [decorate,decoration={brace,amplitude=5pt}]
(-0.3,-0.3) -- (-0.3,2.3) node [black,midway,xshift=-0.6cm]
{\footnotesize $m_1$};
\draw [decorate,decoration={brace,amplitude=5pt}]
(-0.3,2.3) -- (-0.3,4.3) node [black,midway,xshift=-0.6cm]
{\footnotesize $m_2$};
\draw [decorate,decoration={brace,amplitude=5pt}]
(-0.3,4.7) -- (-0.3,6.7) node [black,midway,xshift=-0.6cm]
{\footnotesize $m_3$};
\draw [decorate,decoration={brace,amplitude=5pt}]
(-0.3,6.7) -- (-0.3,8.7) node [black,midway,xshift=-0.6cm]
{\footnotesize $m_4$};
\draw [decorate,decoration={brace,amplitude=5pt}]
(-0.3,8.7) -- (-0.3,11.3) node [black,midway,xshift=-0.6cm]
{\footnotesize $m_5$};
\end{tikzpicture}
\\(a) $Xcon$
&(b) $Xcon|x\leq y$
&(c) $Xcon|y\leq x$
&(d) $Xcon|x=y$
\end{tabular}
\end{center}
\caption{Chain decomposition of $Xcon$}
\end{figure}
\begin{figure}[htb]
\begin{center}
\begin{tabular}{c c c c}

\begin{tikzpicture}[scale=0.5]
\node[draw, circle, inner sep=0pt, minimum size=5pt] at (0,0) (a) {};
\node[draw, circle, inner sep=0pt, minimum size=5pt] at (0,2) (b) {};
\node[draw, circle, inner sep=0pt, minimum size=5pt] at (0,4) (c) {};
\node[draw, circle, inner sep=0pt, minimum size=5pt] at (0,6) (d) {};
\node[draw, circle, inner sep=0pt, minimum size=5pt] at (0,8) (e) {};
\node[draw, circle, inner sep=0pt, minimum size=5pt] at (0,10) (f) {};
\node[draw, circle, inner sep=0pt, minimum size=5pt] at (1,4) (g) {};
\node[draw, circle, inner sep=0pt, minimum size=5pt] at (1,6) (h) {};
\draw[dotted] (a) -- (b) -- (c) -- (d) -- (e) -- (f);
\draw (b) -- (g) -- (d);
\draw (c) -- (h) -- (e);
\draw[dashed] (g) -- (h);
\draw [decorate,decoration={brace,amplitude=5pt}]
(-0.3,-0.3) -- (-0.3,2.3) node [black,midway,xshift=-0.6cm]
{\footnotesize $m_1$};
\draw [decorate,decoration={brace,amplitude=5pt}]
(-0.3,2.3) -- (-0.3,4.3) node [black,midway,xshift=-0.6cm]
{\footnotesize $m_2$};
\draw [decorate,decoration={brace,amplitude=5pt}]
(-0.3,4.3) -- (-0.3,5.7) node [black,midway,xshift=-0.6cm]
{\footnotesize $m_3$};
\draw [decorate,decoration={brace,amplitude=5pt}]
(-0.3,5.7) -- (-0.3,7.7) node [black,midway,xshift=-0.6cm]
{\footnotesize $m_4$};
\draw [decorate,decoration={brace,amplitude=5pt}]
(-0.3,7.7) -- (-0.3,10.3) node [black,midway,xshift=-0.6cm]
{\footnotesize $m_5$};
\end{tikzpicture}
&\begin{tikzpicture}[scale=0.5]
\node[draw, circle, inner sep=0pt, minimum size=5pt] at (0,0) (a) {};
\node[draw, circle, inner sep=0pt, minimum size=5pt] at (0,2) (b) {};
\node[draw, circle, inner sep=0pt, minimum size=5pt] at (0,4) (c) {};
\node[draw, circle, inner sep=0pt, minimum size=5pt] at (0,6) (d) {};
\node[draw, circle, inner sep=0pt, minimum size=5pt] at (0,8) (e) {};
\node[draw, circle, inner sep=0pt, minimum size=5pt] at (0,10) (f) {};
\node[draw, circle, inner sep=0pt, minimum size=5pt] at (1,4) (g) {};
\node[draw, circle, inner sep=0pt, minimum size=5pt] at (1,6) (h) {};
\draw[dotted] (a) -- (b) -- (c) -- (d) -- (e) -- (f);
\draw (b) -- (g) -- (d);
\draw (c) -- (h) -- (e);
\draw (g) -- (h);
\draw [decorate,decoration={brace,amplitude=5pt}]
(-0.3,-0.3) -- (-0.3,2.3) node [black,midway,xshift=-0.6cm]
{\footnotesize $m_1$};
\draw [decorate,decoration={brace,amplitude=5pt}]
(-0.3,2.3) -- (-0.3,4.3) node [black,midway,xshift=-0.6cm]
{\footnotesize $m_2$};
\draw [decorate,decoration={brace,amplitude=5pt}]
(-0.3,4.3) -- (-0.3,5.7) node [black,midway,xshift=-0.6cm]
{\footnotesize $m_3$};
\draw [decorate,decoration={brace,amplitude=5pt}]
(-0.3,5.7) -- (-0.3,7.7) node [black,midway,xshift=-0.6cm]
{\footnotesize $m_4$};
\draw [decorate,decoration={brace,amplitude=5pt}]
(-0.3,7.7) -- (-0.3,10.3) node [black,midway,xshift=-0.6cm]
{\footnotesize $m_5$};
\end{tikzpicture}
&\begin{tikzpicture}[scale=0.5]
\node[draw, circle, inner sep=0pt, minimum size=5pt] at (0,0) (a) {};
\node[draw, circle, inner sep=0pt, minimum size=5pt] at (0,2) (b) {};
\node[draw, circle, inner sep=0pt, minimum size=5pt] at (0,4) (c) {};
\node[draw, circle, inner sep=0pt, minimum size=5pt] at (0,6) (d) {};
\node[draw, circle, inner sep=0pt, minimum size=5pt] at (0,8) (e) {};
\node[draw, circle, inner sep=0pt, minimum size=5pt] at (0,10) (f) {};
\node[draw, circle, inner sep=0pt, minimum size=5pt] at (1,4.5) (g) {};
\node[draw, circle, inner sep=0pt, minimum size=5pt] at (1,5.5) (h) {};
\draw[dotted] (a) -- (b) -- (c) -- (d) -- (e) -- (f);
\draw (c) -- (g) -- (h) -- (d);
\draw [decorate,decoration={brace,amplitude=5pt}]
(-0.3,-0.3) -- (-0.3,2.3) node [black,midway,xshift=-0.6cm]
{\footnotesize $m_1$};
\draw [decorate,decoration={brace,amplitude=5pt}]
(-0.3,2.3) -- (-0.3,4.3) node [black,midway,xshift=-0.6cm]
{\footnotesize $m_2$};
\draw [decorate,decoration={brace,amplitude=5pt}]
(-0.3,4.3) -- (-0.3,5.7) node [black,midway,xshift=-0.6cm]
{\footnotesize $m_3$};
\draw [decorate,decoration={brace,amplitude=5pt}]
(-0.3,5.7) -- (-0.3,7.7) node [black,midway,xshift=-0.6cm]
{\footnotesize $m_4$};
\draw [decorate,decoration={brace,amplitude=5pt}]
(-0.3,7.7) -- (-0.3,10.3) node [black,midway,xshift=-0.6cm]
{\footnotesize $m_5$};
\end{tikzpicture}
&\begin{tikzpicture}[scale=0.5]
\node[draw, circle, inner sep=0pt, minimum size=5pt] at (0,0) (a) {};
\node[draw, circle, inner sep=0pt, minimum size=5pt] at (0,2) (b) {};
\node[draw, circle, inner sep=0pt, minimum size=5pt] at (0,4) (c) {};
\node[draw, circle, inner sep=0pt, minimum size=5pt] at (0,6) (d) {};
\node[draw, circle, inner sep=0pt, minimum size=5pt] at (0,8) (e) {};
\node[draw, circle, inner sep=0pt, minimum size=5pt] at (0,10) (f) {};
\node[draw, circle, inner sep=0pt, minimum size=5pt] at (1,5) (g) {};
\draw[dotted] (a) -- (b) -- (c) -- (d) -- (e) -- (f);
\draw (c) -- (g) -- (d);
\draw [decorate,decoration={brace,amplitude=5pt}]
(-0.3,-0.3) -- (-0.3,2.3) node [black,midway,xshift=-0.6cm]
{\footnotesize $m_1$};
\draw [decorate,decoration={brace,amplitude=5pt}]
(-0.3,2.3) -- (-0.3,4.3) node [black,midway,xshift=-0.6cm]
{\footnotesize $m_2$};
\draw [decorate,decoration={brace,amplitude=5pt}]
(-0.3,4.3) -- (-0.3,5.7) node [black,midway,xshift=-0.6cm]
{\footnotesize $m_3$};
\draw [decorate,decoration={brace,amplitude=5pt}]
(-0.3,5.7) -- (-0.3,7.7) node [black,midway,xshift=-0.6cm]
{\footnotesize $m_4$};
\draw [decorate,decoration={brace,amplitude=5pt}]
(-0.3,7.7) -- (-0.3,10.3) node [black,midway,xshift=-0.6cm]
{\footnotesize $m_5$};
\end{tikzpicture}
\\(a) $Xdis$
&(b) $Xdis|x\leq y$
&(c) $Xdis|y\leq x$
&(d) $Xdis|x=y$
\end{tabular}
\end{center}
\caption{Chain decomposition of $Xdis$}
\end{figure}
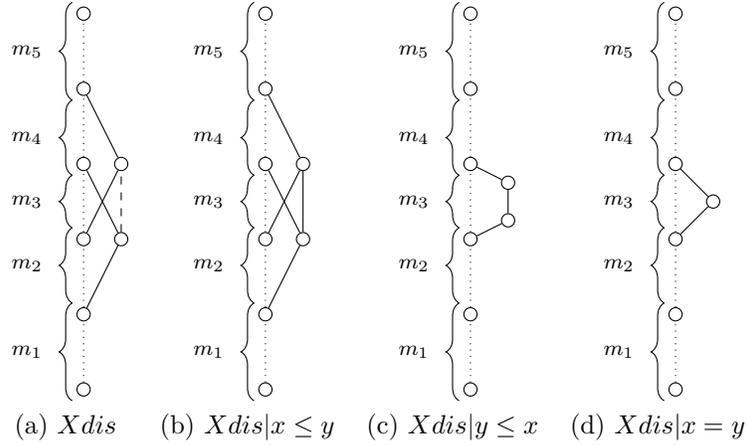
\begin{proposition}
The values of $e(P)$ and $F_P(2)$ of the posets depicted in Figure 4 are given by Table 1.
\end{proposition}
\begin{proof}
The values of the invariants can be computed in two different ways. For simple posets such as the posets depicted in Figure 4, the values of the invariants can be computed combinatorially. For more complicated examples where the values of the invariants cannot be computed combinatorially, a combination of recurrence relations and the formulas in Lemma~\ref{invariants} will still work.
For example, the Tri, Dtri, and Ntri posets are all series parallel and so Lemma~\ref{invariants} can be used to determine the values of the invariants.
The Xcon posets can be inductively reduced (see Figure 7) using the recurrence to previously computed examples and then the Xdis poset can also be reduced (see Figure 8) to previously computed examples.
This demonstrates the utility of working with a family of posets that is closed under the recurrence. We now give the combinatorial computation of the values of the invariants for the posets depicted in Figure 4.

The number of order-preserving maps $f\colon P\to\{1,2\}$ can be determined by first choosing the values of $f$ on the elements not on the maximal chain.
Then the computation of the values of $F_P(2)$ is straightforward.
To compute the values of $e(P)$, it suffices to choose the values of the elements not on the maximal chain.
This is straightfoward for the $Tri$, $Dtri$, and $Ntri$ families.
For the $Xdis$ and $Xcon$ families, let $x$ be the larger element not on the maximal chain and let $y$ be the smaller element not on the maximal chain.
Firstly, consider the $Xdis$ case.
For $x\in\{m_1+1,\ldots,m_1+m_2+m_3+1\}$ there are $m_3+m_4+1$ possible choices for $y$ and for $x=m_1+m_2+m_3+2$ there are $m_3+1$ choices for $y$.
Secondly, consider the $Xcon$ case which is given by taking the previous case an setting $x\geq y$.
This eliminates the $\binom{m_3}{2}$ ways to choose $x,y\in\{m_1+m_2+2,\ldots,m_1+m_2+m_3+1\}$ with $y\geq x$ and the $m_3+1$ ways to choose $y$ with $x=m_1+m_2+m_3+2$. We leave it to the reader to fill in further details.
\end{proof}
\subsection{Ur-Decomposition}
Every poset has a unique decomposition in terms of the Ur-operation. Recall that a poset $P$, $|P| > 1$, is called \textit{prime} if it cannot be expressed as the ordinal sum or disjoint union of two posets. The decomposition of posets into primes by these two operations is known as the series-parallel decomposition. Similarly, a poset $P$, $|P| > 2$, is a \textit{strong prime} if it cannot be expressed as a result of a non-trivial Ur-Operation. Note that a poset is prime if it is a strong prime, but that the converse does not hold. The decomposition of a poset into strong primes by the Ur-Operation gives its \textit{Ur-Decomposition}, a generalization of the series-parallel decomposition.

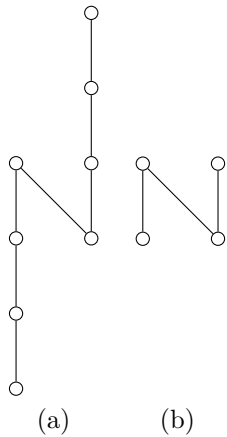
\begin{figure}[htb]
\begin{center}
\begin{tabular}{c c}
\begin{tikzpicture}[scale=0.5]
\node[draw, circle, inner sep=0pt, minimum size=5pt] at (0,0) (a) {};
\node[draw, circle, inner sep=0pt, minimum size=5pt] at (0,2) (b) {};
\node[draw, circle, inner sep=0pt, minimum size=5pt] at (2,2) (c) {};
\node[draw, circle, inner sep=0pt, minimum size=5pt] at (2,0) (d) {};
\node[draw, circle, inner sep=0pt, minimum size=5pt] at (2,4) (e) {};
\node[draw, circle, inner sep=0pt, minimum size=5pt] at (2,6) (f) {};
\node[draw, circle, inner sep=0pt, minimum size=5pt] at (0,-2) (g) {};
\node[draw, circle, inner sep=0pt, minimum size=5pt] at (0,-4) (h) {};
\draw (h) -- (g) -- (a) -- (b) -- (d) -- (c) -- (e) -- (f);
\end{tikzpicture}
&\begin{tikzpicture}[scale=0.5]
\node[draw, circle, inner sep=0pt, minimum size=5pt] at (0,3.85) (a) {};
\node[draw, circle, inner sep=0pt, minimum size=5pt] at (0,5.85) (b) {};
\node[draw, circle, inner sep=0pt, minimum size=5pt] at (2,5.85) (c) {};
\node[draw, circle, inner sep=0pt, minimum size=5pt] at (2,3.85) (d) {};
\node[circle] at (0,0) (e) {};
\draw (a) -- (b) -- (d) -- (c);
\end{tikzpicture}\\
(a)
&(b)
\end{tabular}
\end{center}
\caption{A prime and its corresponding strong prime}
\end{figure}
If $P$ is a poset that has been created by the Ur-Operation, then $P$ will contain subposets that are reducible to a point.
Formally,
\begin{definition}
\label{RAP}
A subset of a poset $S=\{x_k\}_{k=1}^m \subset P$ is reducible to a point (an RAP) when for every $y\in P-S$, either $y\leq\{x_k\}$, $\{x_k\}\leq y$, or $\{x_k\}$ and $y$ are incomparable for all k. An RAP $S$ of $P$ is maximal when it is neither $P$ nor a subset of any other RAP other than $P$.
\end{definition}
Notably, an Ur-operation is an expression of the form $\mathscr{P}[x_k\to P_k]_{k=1}^n$ where each $P_k$ is an RAP.
We will show that if $P$ is a prime then the maximal RAP's partition $P$ and provide a canonical way to decompose $P$ under the Ur-operation.
\begin{lemma}
\label{RAPpartition}
For any prime poset $P$, the maximal RAPs of $P$ partition $P$.
\end{lemma}
\begin{proof}
Recall if $x\in P$ then $\{x\}$ is an RAP of $P$. Then every element of $P$ lies in an RAP, and further in some maximal RAP which contains $\{x\}$.
Now suppose for contradiction that maximal RAPs $A$ and $B$ of $P$ are not disjoint. For distinct $x,y\in P$, let $f(x,y)\in\{``less\ than",``greater\ than",``incomparable"\}$ denote the relation between $x$ and $y$.
If $S$ and $T$ are subsets of $P$ such that $f(s,t)$ is constant over all $s\in S$ and $t\in T$, then we let $f(S,T)$ denote this constant value.
Fix $x\in A\cap B$.
If $y\in P\setminus(A\cup B)$ then $f(y,A) = f(y,x) = f(y,B)$ by the definition of an RAP.
This shows that $A\cup B$ is an RAP.
Then $A\cup B=P$ by the maximality of $A$ and $B$.
For any $x\in A\setminus B$, $y\in A\cap B$, and $z\in B\setminus A$, the definition of an RAP gives that $f(x,y)=f(x,z)=f(y,z)$.
Because this holds for any such $x$, $y$, and $z$, we have that $f(A\setminus B,A\cap B)=f(A\setminus B,B\setminus A)=f(A\cap B,B\setminus A)$ is a well-defined single value.
If this value is $``incomparable"$ then $P=(A\setminus B)+(A\cap B)+(B\setminus A)$.
If this value is $``less\ than"$ then $P=(A\setminus B)\oplus(A\cap B)\oplus(B\setminus A)$.
If this value if $``greater\ than"$ then $P=(B\setminus A)\oplus(A\cap B)\oplus(A\setminus B)$.
In each of these three cases we obtain a contradiction to the fact that $P$ is prime.
\end{proof}
Lemma \ref{RAPpartition} provides a canonical decomposition of a poset $P$ under the Ur-operation.
If $P=P_1+\ldots+P_n$ where $n\geq2$ and where each $P_k$ is indecomposable under disjoint union then write $P=A_n[x_k\to P_k]_{k=1}^n$.
If $P=P_1\oplus\ldots\oplus P_n$ where $n\geq2$ and where each $P_k$ is indecomposable under the ordinal sum then write $P=C_n[x_k\to P_k]_{k=1}^n$.
If $P$ is non-strong prime with maximal RAPs $\{S_k\}_{k=1}^n$ then write $P=Q[x_k\to S_k]$ where $Q$ is the poset of maximal RAPs of $P$.
If $P$ is a strong prime then $P$ is indecomposable under the Ur-operation.
\subsection{Multichains and Poset Reciprocity}
We first present a weak condition on doppelgangers that follows from poset reciprocity.
\begin{lemma}
\label{minusheight}
For all posets $P$, $(-1)^{|P|}F_P(-h(P))=1$ if and only if every element of $P$ is contained in a chain of cardinality $h(P)$.
In particular, doppelganger posets either both satisfy this condition or both fail to satisfy this condition.
\end{lemma}
\begin{proof}
By poset reciprocity, $(-1)^{|P|}F_P(-h(P))$ counts the number of strict order-preserving maps $f\colon P\to[h(P)]$.
If every element of $P$ is contained in a chain of cardinality $h(P)$ then the value of $f$ at each element of $P$ is determined and so $(-1)^{|P|}F_P(-h(P))=1$.
For the converse, suppose that some element $x$ of $P$ is contained in a chain of cardinality smaller than $h(P)$.
Define $f(y)$ to be the cardinality of the largest chain in $P$ with minimal element $y$ and define $g(y)$ to be the cardinality of the largest chain in $P$ with maximal element $y$.
Then $f$ and $g$ are strict order-preserving maps $P\to[h(P)]$ that differ at $x$ and so $(-1)^{|P|}F_P(-h(P))\geq2$.
\end{proof}
Lemma \ref{minusheight} may be used to control potential doppelgangers of sums of chains.
\begin{proposition}
\label{weakkummer}
For positive integers $n$ and $k$, nontrivial doppelgangers of $C_n+\ldots+C_n=C_n\times A_k$ have fewer connected components.
\end{proposition}
\begin{proof}
Let $P=C_n\times A_k$.
Now $h(P)=n$ and every element of $P$ is contained in a chain of cardinality $n$.
If $Q$ is a doppelganger of $P$ then lemma \ref{minusheight} gives that every element of $Q$ is contained in a chain of cardinality $n$.
In particular, every connected component of $Q$ has height at least $n$.
If $Q=Q_1+\ldots+Q_k$ then $kn=|P|=|Q|=|Q_1|+\ldots+|Q_n|\geq h(Q_1)+\ldots+h(Q_n)\geq n+\ldots+n=kn$.
Consequently, these inequalities must be equalities so $|Q_j|=h(Q_j)=n$ for all $1\leq j\leq k$.
This shows that $Q\cong P$.
\end{proof}
In fact, this result may be generalized.
It can be shown that for positive integers $n_1\bigm|\ldots\bigm|n_k$ ($n_i$ divides $n_{i+1}$ for all $1\leq i\leq k-1$), that nontrivial doppelgangers of $P=C_{n_1}+\ldots+C_{n_k}$ have fewer connected components.
If $Q$ is a minimal counterexample to this result then it can be shown by considering the roots of the order polynomials of the connected components of $Q$ that the smallest connected component of $Q$ is a chain of height $1\leq h\leq n_1-1$.
Then $F_P(x)/F_{C_h}(x)$ is an integer-valued polynomial and Kummer's theorem from elementary number theory gives a contradiction.
It would be nice to find a poset-theoretic proof of this result along the lines of the proof of Proposition \ref{weakkummer}.
\section{Acknowledgements}
We would like to thank Professor Hamaker, who presented the problem to us, provided direction for our research, and whose suggestions were invaluable throughout the entire process.
We would also like to thank Professor Morrow, head of the University of Washington Math REU, for making this research a reality for us.
In addition, we would like to thank Sean Griffin for his insightful comments and edits along the way, and Professor Stanley for his guidance later in the process. 
Lastly, we would like to thank Professor Billey, who initially proposed the problem.

\end{document}